\documentclass[english]{extarticle}
\usepackage[T2A,T1]{fontenc}
\usepackage[latin9]{inputenc}
\usepackage{geometry}
\usepackage{mathrsfs}
\usepackage{mathtools}
\usepackage{amsmath}
\usepackage{amsthm}
\usepackage{amssymb}
\usepackage{stackrel}
\usepackage{setspace}

\makeatletter

\DeclareRobustCommand{\cyrtext}{%
  \fontencoding{T2A}\selectfont\def\encodingdefault{T2A}}
\DeclareRobustCommand{\textcyr}[1]{\leavevmode{\cyrtext #1}}

\numberwithin{equation}{section}
\theoremstyle{plain}
\newtheorem{thm}{\protect\theoremname}[section]
\theoremstyle{plain}
\newtheorem{prop}[thm]{\protect\propositionname}
\theoremstyle{definition}
\newtheorem{defn}[thm]{\protect\definitionname}
\theoremstyle{plain}
\newtheorem{lem}[thm]{\protect\lemmaname}
\theoremstyle{remark}
\newtheorem{rem}[thm]{\protect\remarkname}
\theoremstyle{remark}
\newtheorem{claim}[thm]{\protect\claimname}
\ifx\proof\undefined
\newenvironment{proof}[1][\protect\proofname]{\par
	\normalfont\topsep6\p@\@plus6\p@\relax
	\trivlist
	\itemindent\parindent
	\item[\hskip\labelsep\scshape #1]\ignorespaces
}{%
	\endtrivlist\@endpefalse
}
\providecommand{\proofname}{Proof}
\fi
\theoremstyle{plain}
\newtheorem{fact}[thm]{\protect\factname}

\@ifundefined{date}{}{\date{}}
\makeatother

\usepackage{babel}
\usepackage[style=numeric]{biblatex}
\providecommand{\claimname}{Claim}
\providecommand{\definitionname}{Definition}
\providecommand{\factname}{Fact}
\providecommand{\lemmaname}{Lemma}
\providecommand{\propositionname}{Proposition}
\providecommand{\remarkname}{Remark}
\providecommand{\theoremname}{Theorem}

\begin{document}
\title{Derivation of Euler's equations of perfect fluids from von Neumann's
equation with magnetic field}
\maketitle
\begin{center}
IMMANUEL BEN PORAT\footnote{Mathematical Institute, University of Oxford, Oxford OX2 6GG, UK.
Immanuel.BenPorat@maths.ox.ac.uk

Mathematical Reviews subject classification: 35J10, 35Q40. 

 } 
\par\end{center}

\medskip{}

\begin{abstract}
We give a rigorous derivation of the incompressible 2D Euler equation
from the von Neumann equation with an external magnetic field. The
convergence is with respect to the modulated energy functional, and
implies weak convergence in the sense of measures. This is the semi-classical
counterpart of theorem 1.2 in {[}D. Han-Kwan and M. Iacobelli. Proc.
Amer. Math. Soc., \textbf{149} (7):3045--3061 (2021){]}. Our proof
is based on a Gronwall estimate for the modulated energy functional,
which in turn heavily relies on a recent functional inequality due
to {[}S. Serfaty. Duke Math. J. \textbf{169}, 2887--2935 (2020){]}. 
\end{abstract}

\section{Introduction }

\textbf{General background.} In classical mechanics, Newton's second
law of motion (also known as the $N$--body problem) is the following
system of $2\times2\times N$ ODEs 

\begin{equation}
\left\{ \begin{array}{cc}
\dot{x_{k}}(t)=\xi_{k}(t), & x_{k}(0)=x_{k}^{in},\\
\dot{\xi_{k}}(t)=-\frac{1}{\epsilon}\left(\xi_{k}^{\bot}+\frac{1}{N}\underset{j:j\neq k}{\sum}\nabla V(x_{k}(t)-x_{j}(t))\right), & \xi_{k}(0)=\xi_{k}^{in},
\end{array}\right.1\leq k\leq N\label{newton system intro}
\end{equation}

where $x_{k}(t)\in\mathbb{R}^{2}$ (or $x_{k}(t)\in\mathbb{T}^{2}$
where $\mathbb{T}^{2}$ is the $2-$dimensional torus) and $\xi_{k}(t)\in\mathbb{R}^{2}$
are called the position and momentum respectively. From a physical
perspective, the system (\ref{newton system intro}) describes the
dynamics of the positions and momenta of $N$ identical point particles
of unit mass in $\mathbb{R}^{2}$, interacting via an interaction
potential $V$ in the presence of fixed, constant, strong magnetic
field. The parameter $N$ should be thought of as very large, while
the parameter $\epsilon>0$ as very small. We will be mostly concerned
with the magnetic regime, although at times we may refer to the non
magnetic regime as well, for which the corresponding dynamics are
governed by the system 

\begin{equation}
\left\{ \begin{array}{cc}
\dot{x_{k}}(t)=\xi_{k}(t), & x_{k}(0)=x_{k}^{in},\\
\dot{\xi_{k}}(t)=-\frac{1}{N}\underset{1\leq j\leq N,j\neq k}{\sum}\nabla V(x_{k}(t)-x_{j}(t)), & \xi_{k}(0)=\xi_{k}^{in}.
\end{array}\right.1\leq k\leq N.\label{eq:newton system non magnetic}
\end{equation}

On the other hand the Vlasov-Poisson system (specified for 2D) with
the same strong constant magnetic field on $[0,T]\times\mathbb{R}^{2}\times\mathbb{R}^{2}$
$([0,T]\times\mathbb{T}^{2}\times\mathbb{R}^{2})$ reads:

\begin{equation}
\left\{ \begin{array}{c}
\partial_{t}f^{\epsilon}+\xi\cdot\nabla_{x}f^{\epsilon}+\frac{1}{\epsilon}(-\nabla\Phi^{\epsilon}+\xi^{\bot})\cdot\nabla_{\xi}f^{\epsilon}=0,\\
\rho^{\epsilon}=1-\Delta\Phi^{\epsilon}.
\end{array}\right.\label{eq:-8-1}
\end{equation}

The unknown is a time dependent probability density function $f^{\epsilon}(t,x,\xi)$
on $\mathbb{R}^{2}\times\mathbb{R}^{2}$ $(\mathbb{T}^{2}\times\mathbb{R}^{2})$
and $\rho^{\epsilon}$ is a time dependent probability density function
on $\mathbb{R}^{2}$ ($\mathbb{T}^{2}$) defined by 

\[
\rho^{\epsilon}=\rho_{f}^{\epsilon}(t,x)\coloneqq\underset{\mathbb{R}^{2}}{\int}f^{\epsilon}(t,x,\xi)d\xi.
\]

Assuming that $V$ is chosen so that the system (\ref{newton system intro})
is well posed, denote by 
\[
Z_{N}^{t}\coloneqq(x_{1}(t),\xi_{1}(t),...,x_{N}(t),\xi_{N}(t))
\]
 the flow of the system (\ref{newton system intro}) and set $\mu_{Z_{N}}(t)\coloneqq\frac{1}{N}\stackrel[i=1]{N}{\sum}\delta_{Z_{N}^{t}}$.
Unless otherwise stated, we shall focus on the case where the position
variable $x$ is in $\mathbb{R}^{2}$. Denoting by $\mathcal{P}(\mathbb{R}^{d}\times\mathbb{R}^{d})$
the space of probability measures on $\mathbb{R}^{d}\times\mathbb{R}^{d}$,
we have the following classical result due to Klimontovich, which
enables to construct $(\epsilon,N)$-dependent solutions to the Vlasov-Poisson
(\ref{eq:-8-1}) starting from the system (\ref{newton system intro})
\begin{prop}
\begin{flushleft}
For each $N\geq1,\epsilon>0$ the map $t\mapsto\mu_{Z_{N}}(t)$ is
continuous on $[0,T]$ with values in $\mathcal{P}(\mathbb{R}^{2}\times\mathbb{R}^{2})$
equipped with the weak topology, and is a distributional solution
to equation (\ref{eq:-8-1}). 
\par\end{flushleft}

\end{prop}
As customary, we refer to the time dependent probability measure $\mu_{Z_{N}}(t)$
as the \textit{empirical measure} or as \textit{Klimontovich solution
to equation} \textit{(\ref{eq:-8-1})} (the latter name is justified
thanks to the above proposition). In what follows, we restrict our
attention to the case where $V$ is the Green function associated
to the negative Laplacian on $\mathbb{R}^{2}$ (also known as the
2D repulsive Coulomb potential), namely 

\begin{equation}
V(x)\coloneqq-\frac{1}{2\pi}\log(|x|).\label{eq:-1-2}
\end{equation}

Despite the fact that in this case $V$ obviously fails to satisfy
the conditions of the Cauchy-Lipschitz theorem, existence and uniqueness
of solution for the system (\ref{newton system intro}) is still ensured
provided that the positions of the particles are separated initially
(i.e. $x_{k}^{in}\neq x_{l}^{in}$ for $k\neq l$). The Vlasov-Poisson
system is also closely linked to a celebrated model from fluid dynamics,
namely the incompressible 2D Euler equation in vorticity formulation,
which reads 

\begin{equation}
u=(\nabla\psi)^{\bot},\ \omega=\Delta\psi,\ \partial_{t}\omega+\mathrm{div}(u\omega)=0,\label{E}
\end{equation}

where $u:[0,T]\times\mathbb{R}^{2}\rightarrow\mathbb{R}^{2}$, $\omega$
is a scalar field on $[0,T]\times\mathbb{R}^{2}$ and $v^{\bot}\coloneqq(-v_{2},v_{1})$
(if the domain is $\mathbb{T}^{2}$, it is obvious how to adjust the
formulation) . Formal considerations (see e.g section 6 in {[}3{]})
suggest that equation (\ref{E}) (on $\mathbb{T}^{2}$) can be derived
from equation (\ref{eq:-8-1}) in the limit as $\epsilon\rightarrow0$.
In the same work {[}3{]} (especially theorem 6.1), this derivation
has been made rigorous for sufficiently regular/decaying solutions
of equation (\ref{eq:-8-1}). The same problem has been also dealt
within {[}9{]}, where the authors employ compactness methods. Due
to the relation between equation (\ref{eq:-8-1}) and the system (\ref{newton system intro})
provided by Klimontovich's observation, it is natural to seek a derivation
of Euler from Newton's system in the presence\textbackslash absence
of a magnetic field. We stress that due to the singular nature of
Klimontovich solutions, it is far from straightforward to extend the
convergence result of {[}3{]} for Klimontovich solutions of (\ref{eq:-8-1}).
A recent striking functional inequality due to {[}{\small{}19}{]}
(to be discussed in more detail in the next section) has allowed to
overcome the difficulty created due to this singular behavior. This
inequality (to which we refer from now on as Serfaty's inequality)
allowed the same authors to derive the pressureless Euler equation
from Newton's system of ODEs in the absence of a magnetic field (system
(\ref{eq:newton system non magnetic})) and with monokinetic initial
data. The case of non-monokinetic initial data remains a widely open
problem. The derivation of the equation (\ref{E}) from Newton's second
law in the presence of a magnetic field was established in {[}10{]},
again through an argument heavily relying on the above mentioned inequality. 

\medskip{}

\textbf{Main contribution of the current work.} In this work, we focus
on the semi-classical universe. Therefore, we introduce the von Neumann
equation, which is the quantum analogue of Newton's system of ODEs: 

The Cauchy problem for the von Neumann equation with a vector potential
of the form $\frac{1}{2\epsilon}x^{\bot}$ reads

\begin{equation}
i\hbar\partial_{t}R_{N,\epsilon,\hbar}(t)=[\mathscr{H}_{N,\epsilon,\hbar},R_{N,\epsilon,\hbar}(t)],\ R_{N,\epsilon,\hbar}(0)=R_{N,\epsilon,\hbar}^{in},\label{eq:-5-1}
\end{equation}

where $\mathscr{H}_{N}\coloneqq\mathscr{H}_{N,\epsilon,\hbar}$ is
the quantum Hamiltonian defined by the formula

\[
\mathscr{H}_{N}\coloneqq\frac{1}{2}\stackrel[i=1]{N}{\sum}\underset{k=1,2}{\sum}(-i\hbar\partial_{x_{i}^{k}}+\frac{1}{2\epsilon}(x_{i}^{\bot})^{k})^{2}+\frac{1}{2}\stackrel[i=1]{N}{\sum}|x_{i}|^{2}+\frac{1}{N\epsilon}\underset{p<q}{\sum}V_{pq},
\]

and $V_{pq}$ is the multiplication operator corresponding to the
function $V(x_{p}-x_{q})$. As will be clarified in section 6, the
operator $\mathscr{H}_{N}$ can be viewed as a unbounded self-adjoint
operator on $L^{2}(\mathbb{R}^{2N})$. The \textit{Planck constant}
$\hbar>0$ should be viewed as a very small parameter, and thus the
asymptotics in the semi-classical setting are obtained as a triple
limit. For technical reasons we chose to include in the Hamiltonian
a quadratic confining potential of the form $\frac{1}{2}\stackrel[i=1]{N}{\sum}|x_{i}|^{2}$.
We elaborate on the reason for this choice in the next section. The
unknown $R_{N,\epsilon,\hbar}(t)$ is a symmetric density operator
i.e. a bounded operator on $L^{2}(\mathbb{R}^{2N})$ such that 
\[
R=R^{\ast}\geq0,\ \mathrm{trace}(R)=1
\]

and for all $\sigma\in\mathfrak{S}_{N}$ 
\[
U_{\sigma}R_{N,\epsilon,\hbar}U_{\sigma}^{\ast}F_{N}=R_{N,\epsilon,\hbar}F_{N}
\]

where $\mathfrak{S}_{N}$ is the symmetric group on $N$ elements
and where $U_{\sigma}$ is the operator defined for each $F_{N}\in L^{2}(\mathbb{R}^{2N})$
by 

\[
U_{\sigma}F_{N}(x_{1},...,x_{N})\coloneqq F_{N}(x_{\sigma^{-1}(1)},...,x_{\sigma^{-1}(N)}).
\]

In light of the previous discussion, it is natural to seek a derivation
of the incompressible Euler or pressureless Euler equation from the
von Neumann equation, in the presence/absence of a magnetic field
respectively. In order to compare a solution of the von Neumann equation
(which is an operator) with the vorticity solution of the Euler equation
(which is a time dependent function on the Euclidean space), one attaches
to $R_{N,\epsilon,\hbar}(t)$ a time dependent probability density
called the \textit{density of the first marginal of $R_{N,\epsilon,\hbar}(t)$}.
The explicit construction of this density is recalled in the next
section. Thus, deriving Euler from von Neumann reduces to proving
some kind of weak convergence of the density of the first marginal
to the vorticity as the parameters involved grow large or become small
(according to their physical interpretation). The derivation of the
pressureless Euler equation from the von Neumann equation was achieved
in {[}7{]} in the limit as $\frac{1}{N}+\hbar\rightarrow0$. One
of the interesting features of the method of proof of {[}7{]} is
the observation that Serfaty's inequality (which was originally applied
in the context of a classical mean field limit) can be adapted to
the semi-classical regime as well. This observation has also been
utilized in the recent work {[}16{]}, which proves a semi-classical
combined mean field quasineutral limit, which is a semi-classical
version of theorem 1.1 in {[}10{]}. As already mentioned, the second
main result of {[}10{]} is a derivation of equation (\ref{E}) from
the system (\ref{newton system intro}), to which the authors of {[}10{]}
refer to as a combined mean-field and gyrokinetic limit. Our new contribution-
stated precisely in theorem \ref{Main Thm} of the next section- is
a derivation of the incompressible Euler equation (\ref{E}) from
the von Neumann equation (\ref{eq:-5-1}), thereby complementing the
above-mentioned works {[}7{]}, {[}10{]},{[}16{]}. Otherwise put,
we prove a semi-classical combined mean field and gyrokinetic limit.
The recipe for passing from classical to quantum mechanics is summarized
neatly in {[}4{]}:

1. Functions on phase space are replaced by operators on the Hilbert
space of square integrable functions on the underlying configuration
space. 

2. Integration of functions is replaced by the trace of the corresponding
operators.

3. Coordinates $q$ of the configuration space are replaced by multiplication
operator $\widehat{q}$ by the variable $q$, while momentum coordinates
$p$ are replaced by the operator $\widehat{p}=-i\hbar\nabla$ ($\widehat{p}=-i\hbar\nabla+\frac{1}{2\epsilon}x^{\bot}$
in case a magnetic field is included). 

As typical in the theory of mean field limits, the argument in {[}10{]}
rests upon obtaining a Gronwall estimate for a time dependent quantity
which is known to control the weak convergence. This quantity is called
the\textit{ modulated energy}. Our argument is a modification of the
argument leading to a Gronwall estimate on the modulated energy in
{[}10{]} , according to the above mentioned rules 1-3, which is also
the central idea in the semi-classical combined mean field quasineutral
limit obtained in {[}16{]}. Nevertheless, it is important to point
out a few points in which the present work differ from {[}16{]}: 

First, we insist on working on the entire plane $\mathbb{R}^{2}$
rather than the torus $\mathbb{T}^{2}$, since the magnetic vector potential
in question is non-periodic and so it is apriori not obvious how to
even make sense of the modulated energy in the  $\mathbb{T}^{2}$ case. This in turn forces us to
include a quadratic confining potential, in order to make sure that
the quantum Hamiltonian can be viewed as an essentially self-adjoint
operator. Another point which requires some care when working in the
plane is the existence and uniqueness theory of solutions to the incompressible
2D Euler- for example, solutions with square summable velocity field
cannot have a vorticity with a distinguished sign (section 3.1.3 in
{[}14{]})- and therefore such solutions will be inadequate for the
question of interest, in which the vorticity is taken to be a probability
density. Finally, the example cooked up in order to witness the initial
vanishing of the modulated energy has to be adjusted to these new choices.
The utility of each one of the choices we just mentioned will become
clearer in the sequel.

The paper is organized as follows: 

In section 2 a semi-classical version of the modulated energy is introduced
along with other preliminaries, and in section 3 a Gronwall estimate
is established for this quantity. The calculations which lead to this
estimate are more tedious in comparison to the classical setting,
partially due to the fact that quantization gives rise to commutators
of a differential operator with a multiplication operator- which contributes
non zero terms of course. As in {[}7{]},{[}10{]},{[}16{]} this
estimate heavily relies on Serfaty's inequality. Section 4 aims to
explain how weak convergence is implied from this estimate- which
is a simple consequence of a different (yet intimately related) inequality
of Serfaty. In section 5 we construct an explicit example witnessing
the asymptotic vanishing of the initial modulated energy. Finally,
section 6 elaborates on the self-adjointness of the Hamiltonian and
related functional analytic material. 

\section*{Acknowledgement }

This paper is part of the author's work towards a PhD. I would like
to express my deepest graditute towards my supervisor Fran\textcyr{\cyrsdsc}ois
Golse for many fruitful discussions and for carefully reading previous
versions of this manuscript and providing insightful comments and
improvements. I would also like to thank Daniel Han-Kwan and Mikaela
Iacobelli for a helpful correspondence, as well as to Matthew Rosenzweig
for suggesting a few valuable references and offering comments which
improved clarity of exposition. The author declares no conflict of interest. This work was partially supported by \'Ecole Polytechnique and by the research grant "Stability analysis for nonlinear partial differential equations across multiscale applications". The datasets generated during and/or analysed during the current study are available from the corresponding author on reasonable request.

\section{\label{section 2 of 2} Preliminaries and Main Result}

The equation which is to be derived (to which it is customary to refer
to as the ``target equation'') is (\ref{E}). We will work with
an equivalent formulation which reads

\begin{equation}
\partial_{t}u+u\cdot\nabla u=-\nabla P,\ \mathrm{div}(u)=0.\label{eq:E2}
\end{equation}

Taking the divergence of both sides of (\ref{eq:E2}) gives 

\begin{equation}
-\Delta P=\mathrm{div}(u\cdot\nabla u)=\underset{i,j}{\sum}\partial_{x_{i}}u^{j}\partial_{x_{j}}u^{i}\coloneqq\mathfrak{U}.\label{eq:-3-2}
\end{equation}

We elaborate more on the objects related to the von Neumann equation
(\ref{eq:-5-1}). To this aim, let us fix some notations from Hilbert
space theory, adapting mainly the terminology introduced in {[}7{]}.
Set $\mathfrak{H}\coloneqq L^{2}(\mathbb{R}^{2})$ and for each $N\geq1$
denote $\mathfrak{H}_{N}\coloneqq\mathfrak{H}^{\otimes N}\simeq L^{2}(\mathbb{R}^{2N})$.
As customary, $\mathcal{L}(\mathfrak{H})$ stands for the normed space
of bounded linear operators on $\mathfrak{H}$, $\mathcal{L}^{1}(\mathfrak{H})$
stands for the normed space of trace class operators on $\mathfrak{H}$
and $\mathcal{L}^{2}(\mathfrak{H})$ stands for the normed space of
Hilbert Schmidt operators on $\mathfrak{H}$. For each $\sigma\in\mathfrak{S}_{N}$
(where $\mathfrak{S}_{N}$ is the group of permutations on $\{1,...,N\}$)
we define the operator $U_{\sigma}\in\mathcal{L}(\mathfrak{H}_{N})$
by 

\[
(U_{\sigma}\Psi_{N})(x_{1},...,x_{N})\coloneqq\Psi_{N}(x_{\sigma^{-1}(1)},...,x_{\sigma^{-1}(N)}).
\]

With this notation, we denote by $\mathcal{L}_{s}(\mathfrak{H}_{N})$
($\mathcal{L}_{s}^{1}(\mathfrak{H}_{N})$) the set of bounded (trace
class) operators $F_{N}\in\mathcal{L}(\mathfrak{H}_{N})(\mathcal{L}^{1}(\mathfrak{H}_{N}))$
such that $U_{\sigma}F_{N}U_{\sigma}^{\ast}=F_{N}$ for all $\sigma\in\mathfrak{S}_{N}$.We
denote by $\mathcal{D}(\mathfrak{H})$ the set of density operators
on $\mathfrak{H}$, i.e. operators $R\in\mathcal{L}(\mathfrak{H})$
such that $R=R^{\ast}\geq0,\mathrm{trace}_{\mathfrak{H}}(R)=1$. In
addition set $\mathcal{D}_{s}(\mathfrak{H}_{N})\coloneqq\mathcal{D}(\mathfrak{H}_{N})\cap\mathcal{L}_{s}(\mathfrak{H}_{N})$.
Let us recall the notion of marginal
\begin{defn}
\begin{flushleft}
Let $N\geq1$ and $F_{N}\in\mathcal{D}_{s}(\mathfrak{H}_{N})$. For
each $1\leq k\leq N$ we denote by $F_{N:k}\in\mathcal{D}_{s}(\mathfrak{H}_{k})$
the \textit{$k$-th marginal}, which is by definition unique element
of $\mathcal{D}_{s}(\mathfrak{H}_{k})$ such that 
\par\end{flushleft}
\begin{flushleft}
\[
\mathrm{trace}_{\mathfrak{H}_{k}}(A_{k}F_{N:k})=\mathrm{trace}_{\mathfrak{H}_{N}}((A_{k}\otimes I^{\otimes(N-k)})F_{N})
\]
\par\end{flushleft}
\begin{flushleft}
for all $A_{k}\in\mathcal{L}(\mathfrak{H}_{k})$. 
\par\end{flushleft}

\end{defn}
In the sequel we will take $V$ to be as in (\ref{eq:-1-2}). The
unknown $R_{N,\epsilon,\hbar}(t)$ in equation (\ref{eq:-5-1}) is
an element of $\mathcal{D}_{s}(\mathfrak{H}_{N})$. For brevity we
put 
\[
v_{j}^{k}\coloneqq v_{j,\epsilon,\hbar}^{k}\coloneqq-i\hbar\partial_{x_{j}^{k}}+\frac{1}{2\epsilon}(x_{j}^{\bot})^{k}.
\]

The Hamiltonian can be decomposed to the kinetic energy $\mathscr{K}_{N}\coloneqq\mathscr{K}_{N,\epsilon,\hbar}$
and the interaction part $\mathscr{V}_{N}\coloneqq\mathscr{V}_{N,\epsilon}$
which are defined by 
\[
\mathscr{K}_{N}\coloneqq\frac{1}{2}\stackrel[i=1]{N}{\sum}\underset{k=1,2}{\sum}(-i\hbar\partial_{x_{i}^{k}}+\frac{1}{2\epsilon}(x_{i}^{\bot})^{k})^{2}+\frac{1}{2}\stackrel[i=1]{N}{\sum}|x_{i}|^{2},
\]

and 

\[
\mathscr{V}_{N}\coloneqq\frac{1}{N\epsilon}\underset{p<q}{\sum}V_{pq}.
\]

The operator $\mathscr{K}_{N}$ may be viewed as an unbounded essentially
self-adjoint operator in $\mathfrak{H}_{N}$ with domain $C_{0}^{\infty}(\mathbb{R}^{2N})$.
This can be seen through the machinery of quadratic forms methods
(see section 6 of the present work or section 2.2 in {[}23{]} for
details). In section 6 we will prove the estimate 
\[
||\mathscr{V}_{N}\varphi||_{2}^{2}\leq\alpha||\mathscr{K}_{N}\varphi||_{2}^{2}+\beta||\varphi||_{2}^{2}
\]

for some $0<\alpha<1,\beta>0$ and all $\varphi\in C_{0}^{\infty}(\mathbb{R}^{2N})$.
Due to the perturbation theory developed by Kato-Rellich or Kato in
{[}11{]},{[}12{]},{[}13{]} this implies that $\mathscr{H}_{N}$
is essentially self-adjoint, which by Stone's theorem implies that
$e^{-it\mathscr{H}_{N}}$ is a unitary operator commuting with $\mathscr{H}_{N}$-
a fact which is of utter importance . We stress again that some special
care is needed in the case of the 2D Coulomb potential, since originaly
Kato proved that self-adjointness of the kinetic energy is invariant
under perturbations which belong to $L^{\infty}+L^{2}$, whereas $\log(|x|)$
obviously fails to obey this condition. The external potential $\frac{1}{2}\stackrel[i=1]{N}{\sum}|x_{i}|^{2}$
is responsible for handeling this obstcale: as clarified in section
6, it turns out that this term is helpful while establishing the self-adjointness
of $\mathscr{H}_{N}$, since it compensate the divergence of $\log(|x|)$
as $|x|\rightarrow\infty$ . This is one technical difference in comparison
to the work of {[}10{]}. A second related technical difference is
that we confine the discussion to $\mathbb{R}^{2}$ (rather than $\mathbb{T}^{2}$),
since the vector potential $x^{\bot}$ is non-periodic. Possibly,
a gauge invariance principle can be applied in order to overcome this
problem, but this direction remains to be pursued. 

It is well known that the generalized solution to equation (\ref{eq:-5-1})
is given by 
\[
R_{N,\epsilon,\hbar}(t)=e^{-\frac{it\mathscr{H}_{N}}{\hbar}}R_{N,\epsilon,\hbar}^{in}e^{\frac{it\mathscr{H}_{N}}{\hbar}}.
\]
If $R_{N}$ is a density operator then both $R_{N},\sqrt{R_{N}}$
are Hilbert--Schmidt operators on $\mathfrak{H}_{N}$, and denote
their integral kernels by $k(X_{N},Y_{N}),\kappa(X_{N},Y_{N})\in\mathfrak{H}_{N}^{\otimes2}$
respectively. Using that $R_{N}=\sqrt{R_{N}}\sqrt{R_{N}}$ we see
that 

\[
(X_{N},\Xi_{N})\mapsto k(X_{N},X_{N}+\Xi_{N})=\underset{\mathbb{R}^{dN}}{\int}\kappa(X_{N},Z_{N})\overline{\kappa}(X_{N}+\Xi_{N},Z_{N})dZ_{N}
\]

and therefore $(X_{N},\Xi_{N})\mapsto k(X_{N},X_{N}+\Xi_{N})$ defines
an element of $C(\mathbb{R}_{\Xi_{N}}^{dN};L^{1}(\mathbb{R}_{X_{N}}^{dN}))$,
which implies that $\rho[R_{N}](X_{N})\coloneqq X_{N}\mapsto k(X_{N},X_{N})\in L^{1}(\mathbb{R}^{dN})$,
and furthermore that $\rho[R_{N}]\geq0$ and $\underset{\mathbb{R}^{dN}}{\int}\rho[R_{N}](X_{N})dX_{N}=1$.
Hence $\rho[R_{N}]$ is a probability density on $\mathbb{R}^{dN}$
which is moreover symmetric provided $R_{N}$ is symmetric.We will
denote by $\rho_{N:k,\epsilon,\hbar}(t,\cdot)$ the density associated
to $R_{N:k,\epsilon,\hbar}$. Inspired by {[}10{]} we introduce the
time dependent quantity 

\[
\mathcal{E}_{N,\epsilon,\hbar}(t)\coloneqq\mathcal{E}(t)=\frac{\epsilon}{2N}\mathrm{trace}_{\mathfrak{H}_{N}}(\sqrt{R_{N}(t)}(\stackrel[j=1]{N}{\sum}\underset{k}{\sum}(v_{j}^{k}-u^{k}(t,x_{j}))^{2}+\stackrel[j=1]{N}{\sum}|x_{i}|^{2})\sqrt{R_{N}(t)})
\]

\[
+\frac{N-1}{2N}\underset{(\mathbb{R}^{2})^{2}}{\int}V_{12}\rho_{N:2}(t,x,y)dxdy+\frac{1}{2}\underset{\mathbb{R}^{2}}{\int}(V\star(\omega+\epsilon\mathfrak{U}))(\omega+\epsilon\mathfrak{U})(t,x)dx-\underset{\mathbb{R}^{2}}{\int}(V\star(\omega+\epsilon\mathfrak{U}))\rho_{N:1}(t,x)dx
\]

\begin{equation}
\coloneqq\mathcal{E}_{1}(t)+\mathcal{E}_{2}(t).\label{eq:-4-2}
\end{equation}

The quantity $\mathcal{E}(t)$ is a semi-classical rescaled version
of the distance introduced in {[}10{]} formula 1.10. A few remarks
are in order after this definition. First, let us justify that the
2 last integrals on the second line are well defined. Since $V\star\mathfrak{U}(t,\cdot)=\mathrm{div}(V\star(u\cdot\nabla u)(t,\cdot))$,
we see that $V\star\mathfrak{U}(t,\cdot)\in L^{\infty}(\mathbb{R}^{2})$
provided $u\cdot\nabla u(t,\cdot)\in L^{\infty}\cap L^{q}(\mathbb{R}^{2})$
for some $1<q<2$. In addition $\mathfrak{U}(t,\cdot)\in L^{1}(\mathbb{R}^{2})$
assuming $\nabla u(t,\cdot)\in L^{2}(\mathbb{R}^{2})$. Both of these
integrability assumptions ($\nabla u(t,\cdot)\in L^{2}(\mathbb{R}^{2})$
and $u\cdot\nabla u(t,\cdot)\in L^{\infty}\cap L^{q}(\mathbb{R}^{2})$)
are obviously implied by the assumption on $u$ stated in our main
theorem (\ref{Main Thm}) below, and explain why the product $V\star\mathfrak{U}(\omega+\epsilon\mathfrak{U})(t,\cdot)\in L^{1}(\mathbb{R}^{2})$.
The convolution of the logarithm with a summable function is known
to be a function of bounded mean oscillation, or in brief $\mathrm{BMO}$
(see 6.3, (i) in {[}22{]}). In addition, we recall that a $\mathrm{BMO}$
function is in the weighted space $L^{1}((1+|x|)^{-3})$ ({[}22{]},
1.1.4). In light of this reminder we see that in order to make sense
of the above integrals it suffices to require $(1+|x|)^{3}\omega(t,\cdot)\in L^{\infty}(\mathbb{R}^{2})$.
If we require this assumption to hold initially ($t=0$), we can ensure
it is propagated in time for compactly supported initial data. We
elaborate on this last claim in appendix A. As for the first integral
in the definition of $\mathcal{E}(t)$ as well as $\mathcal{E}_{1}(t)$,
we shall impose the following technical assumption which will be needed
in order to justify the fact that both are well defined for all times 

\textbf{Assumption (A).} 
\[
\mathrm{trace}_{\mathfrak{H}_{N}}(\sqrt{R_{N,\epsilon,\hbar}^{in}}(I+\mathscr{K}_{N})^{2}\sqrt{R_{N,\epsilon,\hbar}^{in}})<\infty.
\]
. 

The usefulness of assumption A will become clear while establishing
conservation of energy, as stated in the following 
\begin{lem}
\begin{flushleft}
\textup{\label{ENERGY CON} }Let $R_{N,\epsilon,\hbar}^{in}\in\mathcal{D}_{s}(\mathfrak{H}_{N})$
satisfy assumption (A). Let $R_{N}(t)\coloneqq e^{-\frac{it\mathscr{H}_{N}}{\hbar}}R_{N,\epsilon,\hbar}^{in}e^{\frac{it\mathscr{H}_{N}}{\hbar}}$.
Then 
\par\end{flushleft}
\begin{flushleft}
\[
\frac{\epsilon}{2N}\mathrm{trace}_{\mathfrak{H}_{N}}(\sqrt{R_{N,\epsilon,\hbar}^{in}}(\stackrel[j=1]{N}{\sum}\underset{k}{\sum}(v_{j}^{k})^{2}+\stackrel[i=1]{N}{\sum}|x_{i}|^{2})\sqrt{R_{N,\epsilon,\hbar}^{in}})
\]
\par\end{flushleft}
\[
+\frac{N-1}{2N}\underset{(\mathbb{R}^{2})^{2}}{\int}V_{12}\rho_{N:2}(0,x,y)dxdy
\]

\[
=\frac{\epsilon}{2N}\mathrm{trace}_{\mathfrak{H}_{N}}(\sqrt{R_{N}(t)}(\stackrel[j=1]{N}{\sum}\underset{k}{\sum}(v_{j}^{k})^{2}+\stackrel[i=1]{N}{\sum}|x_{i}|^{2})\sqrt{R_{N}(t)})
\]
\begin{flushleft}
\begin{equation}
+\frac{N-1}{2N}\underset{(\mathbb{R}^{2})^{2}}{\int}V_{12}\rho_{N:2}(t,x,y)dxdy.\label{eq:-6-1}
\end{equation}
\par\end{flushleft}

\end{lem}
We remark that the proof of energy conservation in {[}7{]} does not
consider the presence of a magnetic field. The proof of lemma (\ref{ENERGY CON})
is postponed to section 6. In particular the conservation of energy
includes the (non obvious at first sight) statement that both terms
on the right hand side of (\ref{eq:-6-1}) are separately well defined
for all times $t\in[0,T]$\textbf{.} Assumption (A) also implies in
particular 
\[
\mathrm{trace}_{\mathfrak{H}_{N}}(\sqrt{R_{N,\epsilon,\hbar}^{in}}\left(\stackrel[i=1]{N}{\sum}\underset{k}{\sum}(-i\hbar\partial_{x_{i}}+\frac{1}{2\epsilon}(x_{i}^{\bot})^{k})^{2}\right)\sqrt{R_{N,\epsilon,\hbar}^{in}})<\infty,
\]
which in turn makes the Riesz representation theorem available, thereby
allowing to define a notion of a current, which is the quantum analog
of the first moment of the solution of Vlasov-Poisson system. Denote
by $\lor$ the anticommutator. 
\begin{defn}
\begin{onehalfspace}
\begin{flushleft}
Let $R\in\mathcal{L}^{1}(\mathfrak{H})$ such that $R=R^{\ast}\geq0$
and $\mathrm{trace}_{\mathfrak{H}}(\sqrt{R}((-i\hbar\partial_{x^{k}}+\frac{1}{2\epsilon}(x^{\bot})^{k})^{2})\sqrt{R})<\infty$.
For each $a\in C_{b}(\mathbb{R}^{2})$ and $k=1,2$ we denote by $J_{N:1}^{k}$
the unique signed Radon measure on $\mathbb{R}^{2}$ such that 
\par\end{flushleft}
\begin{flushleft}
\[
\underset{\mathbb{R}^{2}}{\int}a(x)J_{N:1}^{k}=\frac{1}{2}\mathrm{trace}_{\mathfrak{H}}(a\lor((-i\hbar\partial_{x}+\frac{1}{2\epsilon}(x^{\bot})^{k}))R)
\]
\par\end{flushleft}
\begin{flushleft}
The \textit{current} of $R$ is the signed measure valued vector field
$(J_{N:1}^{1},J_{N:1}^{2})$. 
\par\end{flushleft}
\end{onehalfspace}

\end{defn}
We now state Serfaty's remarkable functional inequality, whose considerable
importance while obtaining a Gronwall estimate on the functional $\mathcal{E}(t)$
has already been mentioned. For $X_{N}=(x_{1},...,x_{N})\in(\mathbb{R}^{2})^{N}$
denote $\mu_{X_{N}}\coloneqq\frac{1}{N}\stackrel[i=1]{N}{\sum}\delta_{x_{i}}$. 
\begin{thm}
\begin{flushleft}
\textup{\label{Serfaty } ({[}19{]}, Proposition 1.1)} There is a
constant $C>0$ with the following property. Assume that $\mu\in L^{\infty}(\mathbb{R}{}^{2})$
is a probability density. Then for any $X_{N}\in(\mathbb{R}^{2})^{N}$
s.t. $\forall i\neq j:x_{i}\neq x_{j}$ and any Lipschitz map $\psi:\mathbb{R}^{2}\rightarrow\mathbb{R}^{2}$
we have 
\[
\left|\underset{(\mathbb{R}^{2})^{2}-\triangle}{\int}(\psi(x)-\psi(y))\cdot\nabla V(x-y)(\mu_{X_{N}}-\mu)^{\otimes2}(x,y)dxdy\right|
\]
\par\end{flushleft}
\[
\leq C||\nabla\psi||_{\infty}\left(\mathfrak{f}_{N}(X_{N},\mu)+C(1+||\mu||_{\infty})N^{-\frac{1}{3}}+\frac{\log N}{N}\right)+2C||\psi||_{W^{1,\infty}}(1+||\mu||_{\infty})N^{-\frac{1}{2}},
\]
\begin{flushleft}
where 
\[
\mathfrak{f}_{N}(X_{N},\mu)\coloneqq\underset{(\mathbb{R}^{2})^{2}-\triangle}{\int}V(x-y)(\mu_{X_{N}}-\mu)^{\otimes2}(x,y)dxdy.
\]
 
\par\end{flushleft}

\end{thm}
\begin{rem}
After completing this work, Matthew Rosenzweig drew our attention
to corollary (4.3) in {[}20{]} which provides an improved (and sharp)
version of theorem \ref{Serfaty }, as well as to proposition (3.9)
in {[}17{]} which provides a simpler proof of this improvement. 
\end{rem}
Originally, Serfaty's inequality was used in order to derive the pressureless
Euler system from Newton's second order system of ODEs with monokinetic
initial data. This is a classical mechanics mean field limit type
result. The following lemma serves as a bridge between theorem (\ref{Serfaty })
and the quantum settings which are relevant for us 
\begin{lem}
\begin{flushleft}
\label{Averaging } \textup{({[}7{]}, Lemma 3.5)} Set $\mu(t,\cdot)\coloneqq\omega(t,\cdot)+\epsilon\mathfrak{U}(t,\cdot)$
and 
\[
\mathfrak{F}[R_{N,\epsilon,\hbar}(t),\mu(t,\cdot)]
\]
\par\end{flushleft}
\[
\coloneqq\underset{(\mathbb{R}^{2})^{2}}{\int}V(x-y)\left(\frac{N-1}{N}\rho_{N:2}(t,x,y)+\mu(t,x)\mu(t,y)-2\rho_{N:1}(t,x)\mu(t,y)\right)dxdy
\]

and
\begin{flushleft}
\[
\mathfrak{F}'[R_{N,\epsilon,\hbar}(t),\mu(t,\cdot),u(t,\cdot)]
\]
\par\end{flushleft}
\[
\coloneqq\underset{(\mathbb{R}^{2})^{2}}{\int}(u(t,x)-u(t,y))\nabla V(x-y)
\]

\[
\times\left(\frac{N-1}{N}\rho_{N:2}(t,x,y)+\mu(t,x)\mu(t,y)-2\rho_{N:1}(t,x)\mu(t,y)\right)dxdy.
\]
\begin{flushleft}
Then 
\par\end{flushleft}
\[
\mathfrak{F}[R_{N,\epsilon,\hbar}(t),\mu(t,\cdot)]=\underset{(\mathbb{R}^{2})^{N}}{\int}\mathfrak{f}(X_{N},\mu(t,\cdot))\rho_{\epsilon,\hbar,N}(X_{N},t)dX_{N}
\]

and
\begin{flushleft}
\[
\mathfrak{F}'[R_{N,\epsilon,\hbar},\mu,u]=\underset{(\mathbb{R}^{2})^{N}}{\int}\mathfrak{f}'(X_{N},\mu(t,\cdot))\rho_{\epsilon,\hbar,N}(X_{N}t)dX_{N},
\]
\par\end{flushleft}
\begin{flushleft}
where 
\[
\mathfrak{f}(X_{N},\mu(t,\cdot))\coloneqq\underset{(\mathbb{R}^{2})^{2}-\triangle}{\int}V(x-y)(\mu_{X_{N}}-\mu(t,\cdot))^{\otimes2}(x,y)dxdy,
\]
\par\end{flushleft}
\begin{onehalfspace}
\begin{flushleft}
\[
\mathfrak{f}'(X_{N},\mu(t,\cdot),u)\coloneqq\underset{(\mathbb{R}^{2})^{2}-\triangle}{\int}(u(t,x)-u(t,y))\nabla V(x-y)(\mu_{X_{N}}-\mu(t,\cdot))^{\otimes2}(x,y)dxdy.
\]
\par\end{flushleft}
\end{onehalfspace}

\end{lem}
\begin{onehalfspace}

\end{onehalfspace}Our main theorem is 
\begin{thm}
\label{Main Thm} Let $(\omega,u)\in C^{1}([0,T];C^{0,\alpha}(\mathbb{R}^{2}))\times C^{1}([0,T];C_{\mathrm{loc}}^{1,\alpha}(\mathbb{R}^{2})\cap L^{\infty}(\mathbb{R}^{2}))$
with $\alpha\in(0,1)$ be a solution to (\ref{E}), such that 

1. $\omega(t,\cdot)$ is compactly supported with $\omega\geq0$ and
$||\omega(t,\cdot)||_{1}=1$ and 

2. For all $1<p\leq\infty:$ $\nabla u\in L^{\infty}([0,T];L^{p}(\mathbb{R}^{2}))$.

Let $R_{N,\epsilon,\hbar}(t)\in\mathcal{D}_{s}(\mathfrak{H}_{N})$
be a solution to equation (\ref{eq:-5-1}) such that $R_{N,\epsilon,\hbar}^{in}$
verifies assumption (A). Then $\mathcal{E}_{N,\epsilon,\hbar}(t)\rightarrow0$
as $\frac{1}{N}+\epsilon+\hbar\rightarrow0$, provided $\mathcal{E}_{N,\epsilon,\hbar}(0)\rightarrow0$
as $\frac{1}{N}+\epsilon+\hbar\rightarrow0$. Furthermore, $\rho_{N:1,\epsilon,\hbar}\underset{\frac{1}{N}+\epsilon+\hbar\rightarrow0}{\rightarrow}\omega$
weakly in the sense of measures. 
\end{thm}
\begin{onehalfspace}

\end{onehalfspace}\begin{rem}
The assumption $\mathcal{E}_{N,\epsilon,\hbar}(0)\rightarrow0$ is
reasonably typical, at least for $\epsilon,\hbar$ satisfying some
appropriate relations, as can be seen e.g. through the example constructed
in section (\ref{sec:5 OF CHAP 2}). 
\end{rem}
\begin{rem}
The assumptions on $(\omega,u)$ can be realized for suitable initial
data-see appendix A for more details. 
\end{rem}
The functional $\mathfrak{f}(X_{N},\mu)$ (and thus $\mathcal{E}_{2}(t)$)
is not necessarily a non negative quantity. However, it turns out
that it is bounded from below by a term which vanishes asymptotically.
Therefore $\mathcal{E}(t)\rightarrow0$ implies that the kinetic term
and interaction term vanish separately: $\mathcal{E}_{1}(t)\rightarrow0,\mathcal{E}_{2}(t)\rightarrow0$.
We have the following 
\begin{prop}
\label{low bound }\textup{ ({[}19{]}, Proposition 3.3)} Let $\mu\in L^{\infty}(\mathbb{R}^{2})$
be a probability density. Then

\[
\mathfrak{f}(X_{N},\mu)+\frac{1+||\mu||_{\infty}}{N}+\frac{\log N}{N}\geq0.
\]
 
\end{prop}

\section{\label{sec:Gronwall-Estimate CHAP 2} Gronwall Estimate }

We present here a proof of the asymptotic vanishing of the modulated
energy, as stated in the first part of theorem (\ref{Main Thm}).
In the next section we will explain how to topologize this convergence.
The proof below should be regarded as a formal proof. A careful justification
of the calculation below follows by the eigenfunction expansion method
explained the recent work {[}8{]} and is left as an exercise, since
in our view the formal calculation makes the essence of the argument
more visible. In addition, in some places the dependence on the parameters
$\epsilon,\hbar$ is implicit. During the last stage of the estimate
we will apply Serfaty's inequality, which is valid provided the function
$\mu$ is a bounded probability density. This is not necessarily true
for $\mu=\omega+\epsilon\mathfrak{U}$, and for this reason we will
need to consider 

\[
\mathcal{E}_{N,\epsilon,\hbar}^{\ast}(t)\coloneqq\mathcal{E}^{\ast}(t)
\]

\[
\coloneqq\frac{\epsilon}{2N}\mathrm{trace}_{\mathfrak{H}_{N}}(\sqrt{R_{N}(t)}(\stackrel[j=1]{N}{\sum}\underset{k}{\sum}(v_{j}^{k}-u^{k}(t,x_{j}))^{2}+\stackrel[j=1]{N}{\sum}|x_{i}|^{2})\sqrt{R_{N}(t)})
\]

\[
+\frac{N-1}{2N}\underset{(\mathbb{R}^{2})^{2}}{\int}V_{12}\rho_{N:2}(t,x,y)dxdy+\frac{1}{2}\underset{\mathbb{R}^{2}}{\int}(V\star\omega)\omega(t,x)dx-\underset{\mathbb{R}^{2}}{\int}(V\star\omega)(t,x)\rho_{N:1}(t,x)dx.
\]

\textit{Proof of theorem (\ref{Main Thm}).} \textbf{Step 1. Calculation
of $\dot{\mathcal{E}_{1}}(t)$. }

\begin{onehalfspace}
\[
\frac{d}{dt}\frac{\epsilon}{2N}\mathrm{trace}((\stackrel[j=1]{N}{\sum}\underset{k}{\sum}(u^{k}(t,x_{j})-v_{j}^{k})^{2}+\stackrel[j=1]{N}{\sum}\underset{k}{\sum}|x_{j}^{k}|^{2})R_{N}(t))
\]

\[
=\frac{\epsilon}{2N}\mathrm{trace}\left((\stackrel[j=1]{N}{\sum}\underset{k}{\sum}(u^{k}(t,x_{j})-v_{j}^{k})\lor\partial_{t}u(t,x_{j}))R_{N}(t)\right)
\]

\end{onehalfspace}

\[
+\frac{\epsilon}{2N}\mathrm{trace}\left(((\stackrel[j=1]{N}{\sum}\underset{k}{\sum}(u^{k}(t,x_{j})-v_{j}^{k})^{2}+\stackrel[j=1]{N}{\sum}\underset{k}{\sum}|x_{j}^{k}|^{2})\partial_{t}R_{N}(t)\right)
\]

\begin{onehalfspace}
\[
=\frac{\epsilon}{2N}\mathrm{trace}\left((\stackrel[j=1]{N}{\sum}\underset{k}{\sum}(u^{k}(t,x_{j})-v_{j}^{k})\lor\partial_{t}u(t,x_{j}))R_{N}(t)\right)
\]

\end{onehalfspace}

\[
+\frac{1}{i\hbar}\frac{\epsilon}{2N}\mathrm{trace}\left((\stackrel[j=1]{N}{\sum}\underset{k}{\sum}(u^{k}(t,x_{j})-v_{j}^{k})^{2}+\stackrel[j=1]{N}{\sum}\underset{k}{\sum}|x_{j}^{k}|^{2})\right.
\]

\[
\left.[\frac{1}{2}\stackrel[i=1]{N}{\sum}\underset{l}{\sum}v_{i}^{l}{}^{2}+\frac{1}{2}\stackrel[j=1]{N}{\sum}\underset{k}{\sum}|x_{j}^{k}|^{2}+\frac{1}{N\epsilon}\underset{p<q}{\sum}V_{pq},R_{N}(t)]\right)
\]

\begin{onehalfspace}
\[
=\frac{\epsilon}{2N}\mathrm{trace}\left((\stackrel[j=1]{N}{\sum}\underset{k}{\sum}(u^{k}(t,x_{j})-v_{j}^{k})\lor\partial_{t}u(t,x_{j}))R_{N}(t)\right)
\]

\end{onehalfspace}

\[
+\frac{\epsilon}{2N}\frac{i}{\hbar}\mathrm{trace}\left([\frac{1}{2}\stackrel[i=1]{N}{\sum}\underset{l}{\sum}v_{i}^{l}{}^{2}+\frac{1}{2}\stackrel[j=1]{N}{\sum}\underset{k}{\sum}|x_{j}^{k}|^{2}+\frac{1}{N\epsilon}\underset{p<q}{\sum}V_{pq},\right.
\]

\[
\left.\stackrel[j=1]{N}{\sum}\underset{k}{\sum}(u^{k}(t,x_{j})-v_{j}^{k})^{2}+\stackrel[j=1]{N}{\sum}\underset{k}{\sum}|x_{j}^{k}|^{2}]R_{N}(t)\right).
\]

\begin{onehalfspace}
The second equality is by equation (\ref{eq:-5-1}) and the third
equality follows by tracing by parts ( i.e. the identity $\mathrm{trace}(A[B,C])=-\mathrm{trace}(C[A,B])$).
We start with 

\[
\frac{\epsilon}{2N}\frac{i}{\hbar}\mathrm{trace}\left([\frac{1}{N\epsilon}\underset{p<q}{\sum}V_{pq},\stackrel[j=1]{N}{\sum}\underset{k}{\sum}(u^{k}(t,x_{j})-v_{j}^{k})^{2}+\stackrel[j=1]{N}{\sum}\underset{k}{\sum}|x_{j}^{k}|^{2}]R_{N}(t)\right)
\]

\[
=\frac{\epsilon}{2N}\frac{i}{\hbar}\mathrm{trace}\left(\stackrel[j=1]{N}{\sum}\underset{k}{\sum}(u^{k}(t,x_{j})-v_{j}^{k})\lor[\frac{1}{N\epsilon}\underset{p<q}{\sum}V_{pq},i\hbar\partial_{x_{j}^{k}}]R_{N}(t)\right)
\]

\[
=\frac{1}{2N^{2}}\mathrm{trace}\left(\underset{i\neq j}{\sum}\underset{k}{\sum}(u^{k}(t,x_{i})-v_{i}^{k})\lor(\partial_{x_{i}^{k}}V_{ij})R_{N}(t)\right).
\]

In addition 

\[
\frac{\epsilon}{2N}\frac{i}{\hbar}\mathrm{trace}\left([\frac{1}{2}\stackrel[i=1]{N}{\sum}\underset{l}{\sum}v_{i}^{l}{}^{2}+\frac{1}{2}\stackrel[i=1]{N}{\sum}\underset{l}{\sum}|x_{i}^{l}|^{2},\right.
\]

\end{onehalfspace}

\[
\left.\stackrel[j=1]{N}{\sum}\underset{k}{\sum}(u^{k}(t,x_{j})-v_{j}^{k})^{2}+\stackrel[j=1]{N}{\sum}\underset{k}{\sum}|x_{j}^{k}|^{2}]R_{N}(t)\right)
\]

\begin{onehalfspace}
\[
=\frac{\epsilon}{2N}\frac{i}{\hbar}\frac{1}{2}\stackrel[i=1]{N}{\sum}\mathrm{trace}\left([\underset{l}{\sum}v_{i}^{l}{}^{2}+|x_{i}^{l}|^{2},\underset{k}{\sum}(u^{k}(t,x_{i})-v_{i}^{k})^{2}+|x_{i}^{k}|^{2}]R_{N}(t)\right).
\]

We have 
\[
[v_{i}^{l}{}^{2}+|x_{i}^{l}|^{2},(u^{k}(t,x_{i})-v_{i}^{k})^{2}+|x_{i}^{k}|^{2}]
\]

\end{onehalfspace}

\[
=(u^{k}(t,x_{i})-v_{i}^{k})\vee[v_{i}^{l}{}^{2}+|x_{i}^{l}|^{2},(u^{k}(t,x_{i})-v_{i}^{k})]+[v_{i}^{l}{}^{2},|x_{i}^{k}|^{2}].
\]

\begin{onehalfspace}
We compute 

\[
[v_{i}^{l}{}^{2},(u^{k}(t,x_{i})-v_{i}^{k})]=v_{i}^{l}\vee[v_{i}^{l},(u^{k}(t,x_{i})-v_{i}^{k})]
\]

\[
=v_{i}^{l}\vee(-i\hbar\partial_{x_{i}^{l}}u^{k}-[v_{i}^{l},v_{i}^{k}])
\]

\[
=v_{i}^{l}\vee(-i\hbar\partial_{x_{i}^{l}}u^{k}-[-i\hbar\partial_{x_{i}^{l}}+\frac{1}{2\epsilon}(x_{i}^{\bot})^{l},-i\hbar\partial_{x_{i}^{k}}+\frac{1}{2\epsilon}(x_{i}^{\bot})^{k}])
\]

\[
=v_{i}^{l}\vee(-i\hbar\partial_{x_{i}^{l}}u^{k}+\frac{1}{2\epsilon}i\hbar\partial_{x_{i}^{l}}((x_{i}^{\bot})^{k})-\frac{1}{2\epsilon}i\hbar\partial_{x_{i}^{k}}((x_{i}^{\bot})^{l})),
\]

and 
\[
[|x_{i}^{l}|^{2},(u^{k}(t,x_{i})-v_{i}^{k})]=[|x_{i}^{l}|^{2},-v_{i}^{k}]=[|x_{i}^{l}|^{2},i\hbar\partial_{x_{i}^{k}}]=-i\hbar\partial_{x_{i}^{k}}(|x_{i}^{l}|^{2})=-2i\hbar\delta_{lk}x_{i}^{l},
\]

while 

\[
[v_{i}^{l}{}^{2},|x_{i}^{k}|^{2}]=v_{i}^{l}\vee[v_{i}^{l},|x_{i}^{k}|^{2}]=v_{i}^{l}\vee[-i\hbar\partial_{x_{i}^{l}},|x_{i}^{k}|^{2}]=v_{i}^{l}\vee(-i\hbar\partial_{x_{i}^{l}}(|x_{i}^{k}|^{2}))=-2i\hbar v_{i}^{l}\vee\delta_{lk}x_{i}^{l}.
\]

Hence 

\[
[v_{i}^{l}{}^{2}+|x_{i}^{l}|^{2},(u^{k}(t,x_{i})-v_{i}^{k})^{2}+|x_{i}^{k}|^{2}]
\]

\end{onehalfspace}

\[
=(u^{k}(t,x_{i})-v_{i}^{k})\vee(v_{i}^{l}\vee(-i\hbar\partial_{x_{i}^{l}}u^{k}+\frac{1}{2\epsilon}i\hbar\partial_{x_{i}^{l}}((x_{i}^{\bot})^{k})-\frac{1}{2\epsilon}i\hbar\partial_{x_{i}^{k}}((x_{i}^{\bot})^{l})-2i\hbar\delta_{lk}x_{i}^{l})
\]

\begin{onehalfspace}
\[
-2i\hbar v_{i}^{l}\vee\delta_{lk}x_{i}^{l}.
\]

Put $\mathbf{J}_{kl}=(\partial_{x_{i}^{k}}((x_{i}^{\bot})^{l})-\partial_{x_{i}^{l}}((x_{i}^{\bot})^{k})=(\nabla x^{\bot}-\nabla^{T}x^{\bot})_{kl}$
(so that $\mathbf{J}=2\begin{pmatrix}0 & 1\\
-1 & 0
\end{pmatrix}$). With this notation we find 

\[
\frac{\epsilon}{2N}\frac{i}{\hbar}\frac{1}{2}\stackrel[i=1]{N}{\sum}\mathrm{trace}\left([\underset{l}{\sum}v_{i}^{l}{}^{2}+|x_{i}^{l}|^{2},\underset{k}{\sum}(u^{k}(t,x_{i})-v_{i}^{k})^{2}+|x_{i}^{k}|^{2}]R_{N}(t)\right)
\]

\[
=\frac{\epsilon}{2}\frac{i}{\hbar}\frac{1}{2N}\stackrel[i=1]{N}{\sum}\mathrm{trace}\left((u^{k}(t,x_{i})-v_{i}^{k})\vee\right.
\]

\end{onehalfspace}

\[
\left.(v_{i}^{l}\vee(-i\hbar\partial_{x_{i}^{l}}u^{k}+\frac{1}{2\epsilon}i\hbar\partial_{x_{i}^{l}}((x_{i}^{\bot})^{k})-\frac{1}{2\epsilon}i\hbar\partial_{x_{i}^{k}}((x_{i}^{\bot})^{l})-2i\hbar\delta_{lk}x_{i}^{l})-2i\hbar v_{i}^{l}\vee\delta_{lk}x_{i}^{l})R_{N}(t)\right)
\]

\begin{onehalfspace}
\[
=\frac{\epsilon}{2}\frac{1}{2N}\stackrel[i=1]{N}{\sum}\mathrm{trace}\left((u^{k}(t,x_{i})-v_{i}^{k})\vee(v_{i}^{l}\vee(\partial_{x_{i}^{l}}u^{k}+\frac{1}{2\epsilon}\mathbf{J}_{kl})+2x_{i}^{k})+2v_{i}^{l}\vee x_{i}^{l})R_{N}(t)\right).
\]

So concluding (with Einstein's summation applied for the indices $k,l$) 

\[
\dot{\mathcal{E}_{1}}(t)=\frac{\epsilon}{2N}\mathrm{trace}\left(\stackrel[j=1]{N}{\sum}(u^{k}(t,x_{j})-v_{j}^{k})\lor\partial_{t}u^{k}(x_{j})R_{N}(t)\right)
\]

\end{onehalfspace}

\[
+\frac{1}{2N^{2}}\mathrm{trace}\left(\underset{i\neq j}{\sum}(u^{k}(t,x_{i})-v_{i}^{k})\lor(\partial_{x_{i}^{k}}V_{ij})R_{N}(t)\right)
\]

\begin{onehalfspace}
\[
+\frac{\epsilon}{2}\frac{1}{2N}\stackrel[i=1]{N}{\sum}\mathrm{trace}\left((u^{k}(t,x_{i})-v_{i}^{k})\vee(v_{i}^{l}\vee(\partial_{x_{i}^{l}}u^{k}+\frac{1}{2\epsilon}\mathbf{J}_{kl})+2x_{i}^{k})R_{N}(t)\right)
\]

\[
+\frac{\epsilon}{2}\frac{1}{2N}\stackrel[i=1]{N}{\sum}\mathrm{trace}(2v_{i}^{l}\vee x_{i}^{l}R_{N}(t))
\]

\end{onehalfspace}

\[
=\frac{\epsilon}{2N}\stackrel[i=1]{N}{\sum}\mathrm{trace}\left((u^{k}(t,x_{i})-v_{i}^{k})\lor\right.
\]

\[
\left.(\partial_{t}u^{k}(x_{i})+\frac{1}{2}v_{i}^{l}\vee\partial_{x_{i}^{l}}u^{k}+\frac{1}{N\epsilon}\underset{j:i\neq j}{\sum}\partial_{x_{i}^{k}}V_{ij}+\frac{1}{2}v_{i}^{l}\lor\frac{1}{2\epsilon}\mathbf{J}_{kl}+x_{i}^{k})R_{N}(t)\right)
\]

\begin{onehalfspace}
\begin{equation}
+\frac{\epsilon}{2N}\stackrel[i=1]{N}{\sum}\mathrm{trace}(v_{i}^{l}\vee x_{i}^{l}R_{N}(t)).\label{eq:-13}
\end{equation}

\textbf{Step 2. Calculation of $\dot{\mathcal{E}_{2}}(t)$. }

We start by obtaining an evolution equation for the first marginal
of $\rho_{N,\epsilon,\hbar}$ 
\end{onehalfspace}
\begin{claim}
\begin{onehalfspace}
\begin{flushleft}
\textbf{\label{Evolution equation } $\partial_{t}\rho_{N:1}+\nabla.J_{N:1}=0$
}\textit{in $\mathcal{D}'((0,T)\times\mathbb{R}^{2})$.}\textbf{ }
\par\end{flushleft}
\end{onehalfspace}

\end{claim}
\begin{onehalfspace}
\textit{Proof.}\textsl{ }Let $a\in C_{b}^{1}(\mathbb{R}^{2})$. 
\[
\frac{d}{dt}(\mathrm{trace}(aR_{N:1}(t)))
\]

\end{onehalfspace}

\[
=\frac{1}{i\hbar}\mathrm{trace}\left(a\otimes I_{\mathfrak{H}_{N-1}}[\frac{1}{2}\stackrel[i=1]{N}{\sum}\underset{k}{\sum}(-i\hbar\partial_{x_{i}}+\frac{1}{2\epsilon}(x_{i}^{\bot})^{k})^{2}+\frac{1}{2}\stackrel[j=1]{N}{\sum}\underset{k}{\sum}|x_{j}^{k}|^{2}+\frac{1}{N\epsilon}\underset{p<q}{\sum}V_{pq},R_{N}(t)]\right)
\]

\begin{onehalfspace}
\[
=\frac{i}{\hbar}\mathrm{trace}\left(\frac{1}{2}[\underset{k}{\sum}(-i\hbar\partial_{k}+\frac{1}{2\epsilon}(x_{1}^{\bot})^{k})^{2},a]R_{N:1}(t)\right)
\]

\[
=\frac{i}{\hbar}\mathrm{trace}(\frac{1}{2}(-i\hbar\partial_{k}+\frac{1}{2\epsilon}(x_{1}^{\bot})^{k})\vee(-i\hbar\partial_{k}a)R_{N:1}(t))=\underset{\mathbb{R}^{2}}{\int}\nabla a\cdot J_{N:1}.
\]

Thus 

\[
\partial_{t}\rho_{N:1}+\nabla.J_{N:1}=0.
\]

\end{onehalfspace}
\begin{onehalfspace}
\begin{flushright}
$\square$
\par\end{flushright}
\end{onehalfspace}

\begin{onehalfspace}
One has

\[
\dot{\mathcal{E}_{2}}(t)=\frac{d}{dt}\frac{N-1}{2N}\underset{(\mathbb{R}^{2})^{2}}{\int}V_{12}\rho_{N:2}(t,x,y)dxdy+\frac{1}{2}\frac{d}{dt}\underset{\mathbb{R}^{2}}{\int}(V\star(\omega+\epsilon\mathfrak{U}))(\omega+\epsilon\mathfrak{U})(t,x)dx
\]

\end{onehalfspace}

\[
-\frac{d}{dt}\underset{\mathbb{R}^{2}}{\int}(V\star(\omega+\epsilon\mathfrak{U}))\rho_{N:1}(t,x)dx
\]

\[
=\frac{d}{dt}\frac{N-1}{2N}\underset{(\mathbb{R}^{2})^{2}}{\int}V_{12}\rho_{N:2}(t,x,y)dxdy
\]

\begin{onehalfspace}
\begin{equation}
+\frac{d}{dt}\frac{1}{2}\underset{\mathbb{R}^{2}}{\int}(V\star(\omega+\epsilon\mathfrak{U}))(\omega+\epsilon\mathfrak{U})(t,x)dx-\frac{d}{dt}\underset{\mathbb{R}^{2}}{\int}(V\star(\omega+\epsilon\mathfrak{U}))\rho_{N:1}(t,x)dx.\label{eq:-9-1}
\end{equation}

The second term in the right hand side of equation (\ref{eq:-9-1})
is 

\[
\frac{d}{dt}\frac{1}{2}\underset{\mathbb{R}^{2}}{\int}(V\star(\omega+\epsilon\mathfrak{U}))(\omega+\epsilon\mathfrak{U})(t,x)dx
\]

\[
=\underset{\mathbb{R}^{2}}{\int}(V\star\partial_{t}\omega)\omega(t,x)dx+\underset{\mathbb{R}^{2}}{\int}(V\star\partial_{t}\omega)(\epsilon\mathfrak{U})(t,x)dx
\]

\end{onehalfspace}

\[
+\underset{\mathbb{R}^{2}}{\int}(V\star\omega)(\epsilon\partial_{t}\mathfrak{U})(t,x)dx+\epsilon^{2}\underset{\mathbb{R}^{2}}{\int}(V\star(\partial_{t}\mathfrak{U}))\mathfrak{U}(t,x)dx,
\]

\begin{onehalfspace}
while the third term in the right hand side of equation (\ref{eq:-9-1})
is 

\[
-\frac{d}{dt}\underset{\mathbb{R}^{2}}{\int}(V\star(\omega+\epsilon\mathfrak{U}))\rho_{N:1}(t,x)dx
\]

\end{onehalfspace}

\[
=-\underset{\mathbb{R}^{2}}{\int}(V\star\partial_{t}\omega)\rho_{N:1}(t,x)dx-\underset{\mathbb{R}^{2}}{\int}(V\star\omega)\partial_{t}\rho_{N:1}(t,x)dx
\]

\begin{onehalfspace}
\[
-\epsilon\underset{\mathbb{R}^{2}}{\int}(V\star\partial_{t}\mathfrak{U})\rho_{N:1}(t,x)dx-\epsilon\underset{\mathbb{R}^{2}}{\int}(V\star\mathfrak{U})\partial_{t}\rho_{N:1}(t,x)dx.
\]

So that concluding 
\[
\dot{\mathcal{E}_{2}}(t)=
\]

\[
\frac{d}{dt}\frac{N-1}{2N}\underset{(\mathbb{R}^{2})^{2}}{\int}V_{12}\rho_{N:2}(t,x,y)dxdy
\]

\[
+\underset{\mathbb{R}^{2}}{\int}(V\star\partial_{t}\omega)\omega(t,x)dx+\underset{\mathbb{R}^{2}}{\int}(V\star\partial_{t}\omega)(\epsilon\mathfrak{U})(t,x)dx
\]

\end{onehalfspace}

\[
+\underset{\mathbb{R}^{2}}{\int}(V\star\omega)(\epsilon\partial_{t}\mathfrak{U})(t,x)dx+\epsilon^{2}\underset{\mathbb{R}^{2}}{\int}(V\star(\partial_{t}\mathfrak{U}))\mathfrak{U}(t,x)dx
\]

\[
-\underset{\mathbb{R}^{2}}{\int}(V\star\partial_{t}\omega)\rho_{N:1}(t,x)dx-\underset{\mathbb{R}^{2}}{\int}(V\star\omega)\partial_{t}\rho_{N:1}(t,x)dx
\]

\begin{onehalfspace}
\begin{equation}
-\epsilon\underset{\mathbb{R}^{2}}{\int}(V\star\partial_{t}\mathfrak{U})\rho_{N:1}(t,x)dx-\epsilon\underset{\mathbb{R}^{2}}{\int}(V\star\mathfrak{U})\partial_{t}\rho_{N:1}(t,x).\label{eq:-14}
\end{equation}

Gathering equations (\ref{eq:-13}) and (\ref{eq:-14}) we conclude
that (applying Einstein's summation )

\[
\dot{\mathcal{E}}(t)=
\]

\[
\frac{\epsilon}{2N}\stackrel[i=1]{N}{\sum}\mathrm{trace}\left((u^{k}(t,x_{i})-v_{i}^{k})\lor\right.
\]

\end{onehalfspace}

\[
\left.(\partial_{t}u^{k}(x_{i})+\frac{1}{2}v_{i}^{l}\vee\partial_{x_{i}^{l}}u^{k}+\frac{1}{N\epsilon}\underset{j:i\neq j}{\sum}\partial_{x_{i}^{k}}V_{ij}+\frac{1}{2}v_{i}^{l}\lor\frac{1}{2\epsilon}\mathbf{J}_{kl}+x_{i}^{k})R_{N}(t)\right)
\]

\[
+\frac{\epsilon}{2N}\stackrel[i=1]{N}{\sum}\mathrm{trace}(v_{i}^{l}\vee x_{i}^{l}R_{N}(t))
\]

\begin{onehalfspace}
\[
+\frac{d}{dt}\frac{N-1}{2N}\underset{(\mathbb{R}^{2})^{2}}{\int}V_{12}\rho_{N:2}(t,x,y)dxdy+\underset{\mathbb{R}^{2}}{\int}(V\star\partial_{t}\omega)\omega(t,x)dx
\]

\[
-\underset{\mathbb{R}^{2}}{\int}(V\star\omega)(t,x)\partial_{t}\rho_{N:1}(t,x)dx-\underset{\mathbb{R}^{2}}{\int}(V\star\partial_{t}\omega)\rho_{N:1}(t,x)dx
\]

\end{onehalfspace}

\[
+\epsilon^{2}\underset{\mathbb{R}^{2}}{\int}(V\star(\partial_{t}\mathfrak{U}))\mathfrak{U}(t,x)dx
\]

\begin{onehalfspace}
\[
-\epsilon\underset{\mathbb{R}^{2}}{\int}(V\star\mathfrak{U})\partial_{t}\rho_{N:1}(t,x)dx-\epsilon\underset{\mathbb{R}^{2}}{\int}(V\star\partial_{t}\mathfrak{U})\rho_{N:1}(t,x)dx
\]
\[
+\underset{\mathbb{R}^{2}}{\int}(V\star\partial_{t}\omega)(\epsilon\mathfrak{U})(t,x)dx+\underset{\mathbb{R}^{2}}{\int}(V\star\omega)(\epsilon\partial_{t}\mathfrak{U})(t,x)dx.
\]

\textbf{Step 3. Rearrangement of terms. }By equation (\ref{eq:E2})
we have 

\[
\partial_{t}u^{k}(t,x_{i})+\frac{1}{2}v_{i}^{l}\vee\partial_{x_{i}^{l}}u^{k}(t,x_{i})
\]

\end{onehalfspace}

\[
=-(u\cdot\nabla u+\nabla p)_{k}(t,x_{i})+\frac{1}{2}v_{i}^{l}\vee\partial_{x_{i}^{l}}u^{k}(t,x_{i})
\]

\begin{onehalfspace}
\[
=-u_{l}\partial_{l}u_{k}(t,x_{i})-(\nabla p)_{k}(t,x_{i})+\frac{1}{2}v_{i}^{l}\vee\partial_{x_{i}^{l}}u^{k}(t,x_{i})
\]

\end{onehalfspace}

\[
=(\frac{1}{2}v_{i}^{l}-\frac{1}{2}u_{l})\lor\partial_{l}u_{k}(t,x_{i})-(\nabla p)_{k}(t,x_{i}).
\]

\begin{onehalfspace}
In addition, owing to claim (\ref{Evolution equation }) and noticing
that $u^{\bot}=\nabla(V\star\omega)$ we have 

\[
-\underset{\mathbb{R}^{2}}{\int}(V\star\omega)(t,x)\partial_{t}\rho_{N:1}(t)=-\underset{\mathbb{R}^{2}}{\int}\nabla(V\star\omega)(t,x_{1})\cdot J_{N:1}
\]

\end{onehalfspace}

\[
=-\underset{\mathbb{R}^{2}}{\int}u^{\bot}(t,x_{1})\cdot J_{N:1}
\]

\begin{onehalfspace}
\[
=-\mathrm{trace}\left(\frac{1}{2}(u^{\bot})^{k}\lor(\frac{1}{2\epsilon}(x_{1}^{\bot})^{k}-i\hbar\partial_{k})R_{N:1}\right)=-\frac{1}{2}\mathrm{trace}((u^{\bot})^{k}\lor(v_{1}^{k})R_{N:1}),
\]

and 
\[
\frac{\epsilon}{2N}\stackrel[i=1]{N}{\sum}\mathrm{trace}\left((u^{k}(t,x_{i})-v_{i}^{k})\lor(\frac{1}{2}v_{i}^{l}\lor\frac{1}{2\epsilon}\mathbf{J}_{kl})R_{N}(t)\right)
\]

\[
=\frac{1}{2N}\stackrel[i=1]{N}{\sum}\mathrm{trace}\left((u^{k}(t,x_{i})-v_{i}^{k})\lor(\frac{1}{2}v_{i}^{l}\lor\frac{1}{2}\mathbf{J}_{kl})R_{N}(t)\right)
\]

\[
=\frac{1}{2}\mathrm{trace}\left((u^{k}(t,x_{1})-v_{1}^{k})\lor(\frac{1}{2}v_{1}^{l}\lor\frac{1}{2}\mathbf{J}_{kl})R_{N:1}(t)\right)
\]

\end{onehalfspace}

\[
=\frac{1}{2}\mathrm{trace}\left((u^{k}(t,x_{1})-v_{1}^{k})\lor(\frac{1}{2}\mathbf{J}_{kl}v_{1}^{l})R_{N:1}(t)\right)
\]

\begin{onehalfspace}
\[
=\frac{1}{2}\mathrm{trace}\left((\mathbf{J}_{kl}(u^{k}(t,x_{1})-v_{1}^{k}))\lor(\frac{1}{2}v_{1}^{l})R_{N:1}(t)\right)
\]

\[
=\frac{1}{2}\mathrm{trace}((u^{\bot}(t,x_{1}))^{l}\lor(v_{1}^{l})R_{N:1}(t))-\frac{1}{2}\mathrm{trace}((v_{1}^{\bot})^{l}\vee(v_{1}^{l})R_{N:1}(t)),
\]

so that 
\[
-\underset{\mathbb{R}^{2}}{\int}(V\star\omega)\partial_{t}\rho_{N:1}(t,x)dx+\frac{\epsilon}{2N}\stackrel[i=1]{N}{\sum}\mathrm{trace}\left((u^{k}(t,x_{i})-v_{i}^{k})\lor(\frac{1}{2}v_{i}^{l}\lor\frac{1}{2\epsilon}\mathbf{J}_{kl})R_{N}(t)\right)
\]

\end{onehalfspace}

\[
=-\frac{1}{2}\mathrm{trace}((v_{1}^{\bot})^{l}\vee(v_{1}^{l})R_{N:1}(t)).
\]

\begin{onehalfspace}
In addition, using that $\nabla(V\star\mathfrak{U})=\nabla p$ we
get 

\[
-\epsilon\underset{\mathbb{R}^{2}}{\int}(V\star\mathfrak{U})\partial_{t}\rho_{N:1}(t,x)=-\epsilon\underset{\mathbb{R}^{2}}{\int}\nabla(V\star\mathfrak{U})(t,x_{1})\cdot J_{N:1}
\]

\end{onehalfspace}

\[
=-\epsilon\underset{\mathbb{R}^{2}}{\int}\nabla p(t,x_{1})\cdot J_{N:1}
\]

\begin{onehalfspace}
\[
=-\frac{1}{2}\epsilon\mathrm{trace}((\nabla p)_{k}\lor((\frac{1}{2\epsilon}x_{1}^{\bot})_{k}-i\hbar\partial_{k})R_{N:1})=-\frac{1}{2}\epsilon\mathrm{trace}((\nabla p)_{k}\lor v_{k}^{1}R_{N:1}),
\]

so that $\dot{\mathcal{E}}(t)$ is recast as 

\[
\dot{\mathcal{E}}(t)=
\]

\[
\frac{\epsilon}{2N}\stackrel[i=1]{N}{\sum}\mathrm{trace}\left((u^{k}(t,x_{i})-v_{i}^{k})\lor\right.
\]

\end{onehalfspace}

\[
\left.((\frac{1}{2}v_{i}^{l}-\frac{1}{2}u_{l}(t,x_{i}))\lor\partial_{l}u_{k}(t,x_{i})-(\nabla p)_{k}(t,x_{i})+\frac{1}{N\epsilon}\underset{j:i\neq j}{\sum}\partial_{x_{i}^{k}}V_{ij}+x_{i}^{k})R_{N}(t)\right)
\]

\[
+\frac{\epsilon}{2N}\stackrel[i=1]{N}{\sum}\mathrm{trace}(v_{i}^{l}\vee x_{i}^{l}R_{N}(t))
\]

\begin{onehalfspace}
\[
-\mathrm{trace}((v_{1}^{\bot})^{l}\vee(\frac{1}{2}v_{1}^{l})R_{N:1}(t))
\]

\[
+\frac{d}{dt}\frac{N-1}{2N}\underset{(\mathbb{R}^{2})^{2}}{\int}V_{12}\rho_{N:2}(t,x,y)dxdy+\underset{\mathbb{R}^{2}}{\int}(V\star\partial_{t}\omega)\omega(t,x)dx
\]

\[
-\underset{\mathbb{R}^{2}}{\int}(V\star\partial_{t}\omega)\rho_{N:1}(t,x)dx+\epsilon^{2}\underset{\mathbb{R}^{2}}{\int}(V\star(\partial_{t}\mathfrak{U}))\mathfrak{U}(t,x)dx
\]

\[
-\frac{1}{2}\epsilon\mathrm{trace}((\nabla p)_{k}\lor v_{k}^{1}R_{N:1}(t))-\epsilon\underset{\mathbb{R}^{2}}{\int}(V\star\partial_{t}\mathfrak{U})\rho_{N:1}(t,x)dx
\]
\[
+\underset{\mathbb{R}^{2}}{\int}(V\star\partial_{t}\omega)(\epsilon\mathfrak{U})(t,x)dx+\underset{\mathbb{R}^{2}}{\int}(V\star\omega)(t,x)(\epsilon\partial_{t}\mathfrak{U})(t,x)dx
\]

\[
=-\frac{\epsilon}{2N}\stackrel[i=1]{N}{\sum}\mathrm{trace}\left((u^{k}(t,x_{i})-v_{i}^{k})\lor((\frac{1}{2}u_{l}(t,x_{i})-\frac{1}{2}v_{i}^{l})\lor\partial_{l}u_{k}(t,x_{i}))R_{N}(t)\right)
\]

\[
+\frac{\epsilon}{2N}\stackrel[i=1]{N}{\sum}\mathrm{trace}\left(u^{k}(t,x_{i})\vee(-(\nabla p)_{k}(t,x_{i})+\frac{1}{N\epsilon}\underset{j:i\neq j}{\sum}\partial_{x_{i}^{k}}V_{ij})R_{N}(t)\right)
\]

\[
+\frac{1}{2N^{2}}\stackrel[i=1]{N}{\sum}\mathrm{trace}\left(-v_{i}^{k}\vee\underset{j:i\neq j}{\sum}\partial_{x_{i}^{k}}V_{ij}R_{N}(t)\right)-\frac{1}{2}\mathrm{trace}((v_{1}^{\bot})^{k}\vee(v_{1}^{k})R_{N:1}(t))
\]

\end{onehalfspace}

\[
+\frac{d}{dt}\frac{N-1}{2N}\underset{(\mathbb{R}^{2})^{2}}{\int}V_{12}\rho_{N:2}(t,x,y)dxdy
\]

\begin{onehalfspace}
\[
-\underset{\mathbb{R}^{2}}{\int}(V\star\partial_{t}\omega)\rho_{N:1}(t,x)dx+\underset{\mathbb{R}^{2}}{\int}(V\star\partial_{t}\omega)\omega(t,x)dx+\underset{\mathbb{R}^{2}}{\int}(V\star\partial_{t}\omega)(\epsilon\mathfrak{U})(t,x)dx
\]

\[
-\epsilon\underset{\mathbb{R}^{2}}{\int}(V\star\partial_{t}\mathfrak{U})\rho_{N:1}(t,x)dx
\]

\[
+\underset{\mathbb{R}^{2}}{\int}(V\star\omega)(\epsilon\partial_{t}\mathfrak{U})(t,x)+\epsilon^{2}\underset{\mathbb{R}^{2}}{\int}(V\star(\partial_{t}\mathfrak{U}))\mathfrak{U}(t,x)
\]

\[
+\frac{\epsilon}{2N}\stackrel[i=1]{N}{\sum}\mathrm{trace}((u^{k}(t,x_{i})-v_{i}^{k})\lor x_{i}^{k}R_{N}(t))+\frac{\epsilon}{2N}\stackrel[i=1]{N}{\sum}\mathrm{trace}(v_{i}^{l}\vee x_{i}^{l}R_{N}(t)).
\]

We now apply conservation of energy. 
\end{onehalfspace}
\begin{claim}
\begin{onehalfspace}
\label{Claim } \textit{It holds that }
\[
-\frac{1}{2N^{2}}\stackrel[i=1]{N}{\sum}\mathrm{trace}\left(v_{i}^{k}\vee\underset{j:i\neq j}{\sum}\partial_{x_{i}^{k}}V_{ij}R_{N}(t)\right)-\frac{1}{2}\mathrm{trace}((v_{1}^{\bot})^{k}\vee(v_{1}^{k})R_{N:1}(t))
\]
\[
+\frac{d}{dt}\frac{N-1}{2N}\underset{(\mathbb{R}^{2})^{2}}{\int}V_{12}\rho_{N:2}(t)=0.
\]
. 
\end{onehalfspace}
\end{claim}
\begin{onehalfspace}
\textit{Proof. }We compute the time derivative of the right hand side
of equation (\ref{ENERGY CON}). 
\[
\frac{d}{dt}\frac{\epsilon}{2N}\mathrm{trace}_{\mathfrak{H}_{N}}\left((\stackrel[j=1]{N}{\sum}\underset{k}{\sum}(v_{j}^{k})^{2}+\stackrel[j=1]{N}{\sum}|x_{j}|^{2})R_{N}(t)\right)
\]

\end{onehalfspace}

\[
=\frac{\epsilon}{2N}\frac{1}{i\hbar}\mathrm{trace}_{\mathfrak{H}_{N}}\left((\stackrel[j=1]{N}{\sum}\underset{k}{\sum}(v_{j}^{k})^{2}+\stackrel[j=1]{N}{\sum}|x_{j}|^{2})[\mathscr{H}_{N},R_{N}(t)]\right)
\]

\begin{onehalfspace}
\[
=-\frac{\epsilon}{2N}\frac{1}{i\hbar}\stackrel[j=1]{N}{\sum}\mathrm{trace}_{\mathfrak{H}_{N}}\left(R_{N}(t)[\frac{1}{2}\stackrel[i=1]{N}{\sum}\underset{l}{\sum}v_{i}^{l}{}^{2}+\frac{1}{2}\stackrel[i=1]{N}{\sum}\underset{l}{\sum}|x_{i}^{l}|{}^{2}+\frac{1}{N\epsilon}\underset{p<q}{\sum}V_{pq},\underset{k}{\sum}(v_{j}^{k})^{2}+|x_{j}^{k}|^{2}]\right).
\]

First 

\[
-\frac{\epsilon}{2N}\frac{1}{i\hbar}\stackrel[j=1]{N}{\sum}\stackrel[i=1]{N}{\sum}\mathrm{trace}_{\mathfrak{H}_{N}}\left(R_{N}(t)[\frac{1}{2}v_{i}^{l}{}^{2}+\frac{1}{2}|x_{i}^{l}|{}^{2},v_{j}^{k}{}^{2}+|x_{j}^{k}|^{2}]\right)
\]

\end{onehalfspace}

\[
=-\frac{\epsilon}{2N}\frac{1}{i\hbar}\stackrel[i=1]{N}{\sum}\mathrm{trace}_{\mathfrak{H}_{N}}\left(R_{N}(t)[\frac{1}{2}v_{i}^{l}{}^{2}+\frac{1}{2}|x_{i}^{l}|{}^{2},v_{i}^{k}{}^{2}+|x_{i}^{k}|^{2}]\right)
\]

\begin{onehalfspace}
\[
=-\frac{\epsilon}{4N}\frac{1}{i\hbar}\stackrel[i=1]{N}{\sum}\mathrm{trace}_{\mathfrak{H}_{N}}\left(R_{N}(t)(v_{i}^{l}\lor(v_{i}^{k}\vee[v_{i}^{l},v_{i}^{k}])-2i\hbar v_{i}^{l}\vee\delta_{lk}x_{i}^{l}+2i\hbar v_{i}^{k}\vee\delta_{lk}x_{i}^{l})\right)
\]

\[
=\frac{\epsilon}{4N}\stackrel[i=1]{N}{\sum}\mathrm{trace}_{\mathfrak{H}_{N}}\left(R_{N}(t)(v_{i}^{l}\lor(v_{i}^{k}\lor(\frac{1}{2\epsilon}\partial_{x_{i}^{l}}((x_{i}^{\bot})^{k})-\frac{1}{2\epsilon}\partial_{x_{i}^{k}}((x_{i}^{\bot})^{l})+2v_{i}^{l}\vee x_{i}^{l}-2v_{i}^{k}\vee x_{i}^{k}\right)
\]

\[
=-\frac{1}{4N}\stackrel[i=1]{N}{\sum}\mathrm{trace}_{\mathfrak{H}_{N}}(R_{N}(t)v_{i}^{l}\lor(v_{i}^{k}\lor\frac{1}{2}\mathbf{J}_{kl}))
\]

\end{onehalfspace}

\[
=-\frac{1}{2N}\stackrel[i=1]{N}{\sum}\mathrm{trace}_{\mathfrak{H}_{N}}(R_{N}(t)v_{i}^{l}\lor(v_{i}{}^{\bot})^{l})=-\frac{1}{2}\mathrm{trace}(R_{N:1}(t)v_{1}^{k}\lor(v_{1}{}^{\bot})^{k}).
\]

\begin{onehalfspace}
Second 

\[
-\frac{\epsilon}{2N}\frac{1}{i\hbar}\stackrel[j=1]{N}{\sum}\mathrm{trace}_{\mathfrak{H}_{N}}\left(R_{N}(t)[\frac{1}{N\epsilon}\underset{p<q}{\sum}V_{pq},\underset{k}{\sum}(v_{j}^{k})^{2}+|x_{j}^{k}|^{2}]\right)
\]

\[
=-\frac{1}{2N^{2}}\frac{1}{i\hbar}\stackrel[j=1]{N}{\sum}\underset{p<q}{\sum}\mathrm{trace}_{\mathfrak{H}_{N}}(R_{N}(t)[V_{pq},(v_{j}^{k})^{2}])
\]

\end{onehalfspace}

\[
=-\frac{1}{2N^{2}}\frac{1}{i\hbar}\stackrel[j=1]{N}{\sum}\underset{p<q}{\sum}\mathrm{trace}_{\mathfrak{H}_{N}}(R_{N}(t)v_{j}^{k}\lor[V_{pq},v_{j}^{k}])
\]

\begin{onehalfspace}
\[
=-\frac{1}{2N^{2}}\stackrel[j=1]{N}{\sum}\underset{p<q}{\sum}\mathrm{trace}_{\mathfrak{H}_{N}}(R_{N}(t)v_{j}^{k}\lor\partial_{x_{j}^{k}}V_{pq})
\]

\end{onehalfspace}

\[
=-\frac{1}{2N^{2}}\stackrel[i=1]{N}{\sum}\mathrm{trace}_{\mathfrak{H}_{N}}(R_{N}(t)v_{i}^{k}\lor\underset{j:i\neq j}{\sum}\partial_{x_{i}^{k}}V_{ij}).
\]

\begin{onehalfspace}
Since the left hand side of equation (\ref{ENERGY CON}) is constant
in time we conclude 

\[
-\frac{1}{2N^{2}}\stackrel[i=1]{N}{\sum}\mathrm{trace}_{\mathfrak{H}_{N}}(R_{N}(t)v_{i}^{k}\lor\underset{j:i\neq j}{\sum}\partial_{x_{i}^{k}}V_{ij})-\frac{1}{2}\mathrm{trace}(R_{N:1}(t)v_{1}^{k}\lor(v_{1}{}^{\bot})^{k})
\]

\end{onehalfspace}

\[
+\frac{d}{dt}\frac{N-1}{2N}\underset{(\mathbb{R}^{2})^{2}}{\int}V_{12}\rho_{N:2}(t)=0.
\]

\begin{onehalfspace}
\begin{flushright}
$\square$
\par\end{flushright}
\end{onehalfspace}

\begin{onehalfspace}
In view of claim (\ref{Claim }) $\dot{\mathcal{E}}(t)$ writes 

\[
\dot{\mathcal{E}}(t)
\]

\[
=-\frac{\epsilon}{2N}\stackrel[i=1]{N}{\sum}\mathrm{trace}\left((u^{k}(t,x_{i})-v_{i}^{k})\lor((\frac{1}{2}u_{l}-\frac{1}{2}v_{i}^{l})\lor\partial_{l}u_{k})R_{N}(t)\right)
\]

\[
+\frac{\epsilon}{2N}\stackrel[i=1]{N}{\sum}\mathrm{trace}\left(u^{k}(t,x_{i})\lor(-(\nabla p)_{k}+\frac{1}{N\epsilon}\underset{j:i\neq j}{\sum}\partial_{x_{i}^{k}}V_{ij})R_{N}(t)\right)
\]

\[
-\underset{\mathbb{R}^{2}}{\int}(V\star\partial_{t}\omega)\rho_{N:1}(t,x)+\underset{\mathbb{R}^{2}}{\int}(V\star\partial_{t}\omega)\omega(t,x)dx
\]

\[
+\underset{\mathbb{R}^{2}}{\int}(V\star\partial_{t}\omega)(\epsilon\mathfrak{U})(t,x)dx-\epsilon\underset{\mathbb{R}^{2}}{\int}(V\star\partial_{t}\mathfrak{U})\rho_{N:1}(t,x)dx
\]

\end{onehalfspace}

\[
+\underset{\mathbb{R}^{2}}{\int}(V\star\omega)(\epsilon\partial_{t}\mathfrak{U})(t,x)dx
\]

\begin{onehalfspace}
\begin{equation}
+\epsilon^{2}\underset{\mathbb{R}^{2}}{\int}(V\star(\partial_{t}\mathfrak{U}))\mathfrak{U}(t,x)dx+\frac{\epsilon}{2N}\stackrel[i=1]{N}{\sum}\mathrm{trace}(u^{k}(t,x_{i})\lor x_{i}^{k}R_{N}(t)).\label{eq:arrange of terms}
\end{equation}

\textbf{Step 4. Accessing Serfaty's Inequality. }

Note that 

\[
-\underset{\mathbb{R}^{2}}{\int}(V\star\partial_{t}\omega)\rho_{N:1}(t,x)dx+\underset{\mathbb{R}^{2}}{\int}(V\star\partial_{t}\omega)(x)\omega(t,x)dx
\]

\end{onehalfspace}

\[
+\underset{\mathbb{R}^{2}}{\int}(V\star\partial_{t}\omega)(\epsilon\mathfrak{U})(t,x)dx
\]

\begin{onehalfspace}
\[
=\underset{(\mathbb{R}^{2})^{2}}{\int}V(x-y)\partial_{t}\omega(y)(\omega+\epsilon\mathfrak{U}-\rho_{N:1}(t))(x)dydx.
\]

and 

\[
-\epsilon\underset{\mathbb{R}^{2}}{\int}(V\star\partial_{t}\mathfrak{U})\rho_{N:1}(t,x)dx+\underset{\mathbb{R}^{2}}{\int}(V\star\omega)(\epsilon\partial_{t}\mathfrak{U})(t,x)dx
\]

\end{onehalfspace}

\[
+\epsilon^{2}\underset{\mathbb{R}^{2}}{\int}(V\star(\partial_{t}\mathfrak{U}))\mathfrak{U}(t,x)dx
\]

\begin{onehalfspace}
\[
=\epsilon\underset{(\mathbb{R}^{2})^{2}}{\int}V(x-y)\partial_{t}\mathfrak{U}(y)(\omega+\epsilon\mathfrak{U}-\rho_{N:1})(x)dydx.
\]

Denote $\mu\coloneqq\omega+\epsilon\mathfrak{U}$. We have the identity 

\[
\underset{(\mathbb{R}^{2})^{2}}{\int}(u(t,x)-u(t,y))\nabla V(x-y)(\frac{N-1}{N}\rho_{N:2}(t)(x,y)+\mu(t,x)\mu(t,y)-2\rho_{N:1}(x)\mu(t,y))dxdy
\]

\[
=\frac{N-1}{N}\underset{(\mathbb{R}^{2})^{2}}{\int}(u(t,x)-u(t,y))\nabla V(x-y)\rho_{N:2}(t)(x,y)dxdy-2\epsilon\underset{\mathbb{R}^{2}}{\int}u(t,x)\cdot\nabla p(t,x)\rho_{N:1}(t,x)dx
\]

\begin{equation}
+2\underset{\mathbb{R}^{2}}{\int}\nabla(V\star\omega u)(t,x)(\rho_{N:1}-\mu)(x)dx+2\epsilon\underset{\mathbb{R}^{2}}{\int}\nabla(V\star u\mathfrak{U})(t,x)(\rho_{N:1}-\mu)(x)dx.\label{eq:-12-1}
\end{equation}

By equation (\ref{eq:arrange of terms}) we have 
\[
\dot{\mathcal{E}}(t)
\]

\[
=-\frac{\epsilon}{2N}\stackrel[i=1]{N}{\sum}\mathrm{trace}\left((u^{k}(t,x_{i})-v_{i}^{k})\lor((\frac{1}{2}u_{l}-\frac{1}{2}v_{i}^{l})\lor\partial_{l}u_{k})R_{N}(t)\right)
\]

\[
+\frac{\epsilon}{2N}\stackrel[i=1]{N}{\sum}\mathrm{trace}\left(u^{k}(t,x_{i})\lor(-(\nabla p)_{k}+\frac{1}{N\epsilon}\underset{j:i\neq j}{\sum}\partial_{x_{i}^{k}}V_{ij})R_{N}(t)\right)
\]

\[
+\underset{(\mathbb{R}^{2})^{2}}{\int}V(x-y)\partial_{t}\omega(y)(\omega+\epsilon\mathfrak{U}-\rho_{N:1}(t))(x)dydx
\]

\[
+\epsilon\underset{(\mathbb{R}^{2})^{2}}{\int}V(x-y)\partial_{t}\mathfrak{U}(y)(\omega+\epsilon\mathfrak{U}-\rho_{N:1})(x)dydx
\]

\end{onehalfspace}

\[
+\frac{\epsilon}{2N}\stackrel[i=1]{N}{\sum}\mathrm{trace}(u^{k}(t,x_{i})\lor x_{i}^{k}R_{N}(t))
\]

\begin{onehalfspace}
\[
=-\frac{\epsilon}{2N}\stackrel[i=1]{N}{\sum}\mathrm{trace}\left((u^{k}(t,x_{i})-v_{i}^{k})\lor((\frac{1}{2}u_{l}-\frac{1}{2}v_{i}^{l})\lor\partial_{l}u_{k})R_{N}(t)\right)
\]

\[
+\frac{\epsilon}{2N}\stackrel[i=1]{N}{\sum}\mathrm{trace}\left(u^{k}(t,x_{i})\lor(-(\nabla p)_{k}+\frac{1}{N\epsilon}\underset{j:i\neq j}{\sum}\partial_{x_{i}^{k}}V_{ij})R_{N}(t)\right)
\]

\[
+\underset{(\mathbb{R}^{2})^{2}}{\int}\nabla V(x-y)\cdot u(t,y)\omega(t,y)(\rho_{N:1}(t)-\omega-\epsilon\mathfrak{U})(x)dydx
\]

\[
+\epsilon\underset{\mathbb{R}^{2}}{\int}\partial_{t}p(x)(\omega+\epsilon\mathfrak{U}-\rho_{N:1})(x)dx+\frac{\epsilon}{2N}\stackrel[i=1]{N}{\sum}\mathrm{trace}(u^{k}(t,x_{i})\lor x_{i}^{k}R_{N}(t)).
\]

Inserting equation (\ref{eq:-12-1}) in the above yields 

\[
\dot{\mathcal{E}}(t)=-\frac{\epsilon}{2N}\stackrel[i=1]{N}{\sum}\mathrm{trace}((u^{k}(t,x_{i})-v_{i}^{k})\lor((\frac{1}{2}u_{l}-\frac{1}{2}v_{i}^{l})\lor\partial_{l}u_{k})R_{N}(t))
\]

\[
+\frac{1}{2}\underset{(\mathbb{R}^{2})^{2}}{\int}(u(t,x)-u(t,y))\nabla V(x-y)(\frac{N-1}{N}\rho_{N:2}(t)(x,y)+\mu(t,x)\mu(t,y)-2\rho_{N:1}(x)\mu(t,y))dxdy
\]

\[
-\epsilon\underset{\mathbb{R}^{2}}{\int}\nabla(V\star u\mathfrak{U})(t,x)(\rho_{N:1}-\mu)(x)dx
\]

\[
+\epsilon\underset{\mathbb{R}^{2}}{\int}\partial_{t}p(x)(\omega+\epsilon\mathfrak{U}-\rho_{N:1})(x)dx
\]

\[
+\frac{\epsilon}{N}\stackrel[i=1]{N}{\sum}\mathrm{trace}(x_{i}^{k}u^{k}(t,x_{i})R_{N}(t))\coloneqq\stackrel[i=1]{5}{\sum}J_{i}.
\]

\textbf{Step 5. Gronwall inequality and conclusion.} We estimate each
of the $J_{i}$ separately: 

\textit{Estimating $J_{1}$. }

\[
|J_{1}|
\]

\end{onehalfspace}

\[
\leq\frac{\epsilon}{2N}\stackrel[i=1]{N}{\sum}\left|\mathrm{trace}((u^{k}(t,x_{i})-v_{i}^{k})\lor((\frac{1}{2}u^{l}(t,x_{i})-\frac{1}{2}v_{i}^{l})\lor\partial_{l}u^{k}(t,x_{i}))R_{N}(t))\right|
\]

\[
\leq\frac{\epsilon}{N}||\nabla u||_{L^{\infty}L^{\infty}}\stackrel[i=1]{N}{\sum}||(u^{k}-v_{i}^{k})R_{N}^{\frac{1}{2}}(t)||_{2}||(u_{l}-v_{i}^{l})R_{N}^{\frac{1}{2}}(t)||_{2}
\]

\begin{onehalfspace}
\begin{equation}
\leq||\nabla u||_{L^{\infty}L^{\infty}}\mathcal{E}_{1}(t).\label{eq:-15}
\end{equation}

\textit{Estimating $J_{2}$. }

By lemma (\ref{Averaging }) we can write 

\[
J_{2}=\frac{1}{2}\underset{(\mathbb{R}^{2})^{N}}{\int}\underset{(\mathbb{R}^{2})^{2}-\triangle}{\int}(u(t,x)-u(t,y))\nabla V(x-y)(\mu_{X_{N}}-\omega(t,\cdot))^{\otimes2}(x,y)dxdy\rho_{N,\epsilon,\hbar}(t)dX_{N}
\]

\[
+\frac{1}{2}\epsilon\underset{(\mathbb{R}^{2})^{2}}{\int}(u(t,x)-u(t,y))\nabla V(x-y)\mathfrak{R}(t,x,y)dxdy
\]

\end{onehalfspace}

\[
-\epsilon\underset{(\mathbb{R}^{2})^{2}}{\int}(u(t,x)-u(t,y))\nabla V(x-y)\rho_{N:1}(t,x)\mathfrak{U}(t,y)dxdy,
\]

\begin{onehalfspace}
where 

\[
\mathfrak{R}(t,x,y)\coloneqq\mathfrak{U}(t,y)\omega(t,x)+\mathfrak{U}(t,x)\omega(t,y)+\epsilon\mathfrak{U}(t,x)\mathfrak{U}(t,y).
\]

We observe the following bounds 
\[
\left|\epsilon\underset{(\mathbb{R}^{2})^{2}}{\int}(u(t,x)-u(t,y))\nabla V(x-y)\rho_{N:1}(t,x)\mathfrak{U}(t,y)dxdy\right|
\]

\end{onehalfspace}

\[
\leq\epsilon(||\nabla V\star(u\mathfrak{U})||_{L^{\infty}L^{\infty}}+||u\nabla V\star\mathfrak{U}(t,x)||_{L^{\infty}L^{\infty}})\lesssim\epsilon,
\]

\begin{onehalfspace}
and 

\[
\left|\frac{1}{2}\epsilon\underset{(\mathbb{R}^{2})^{2}}{\int}(u(t,x)-u(t,y))\nabla V(x-y)\mathfrak{R}(t,x,y)dxdy\right|
\]

\end{onehalfspace}

\[
=\epsilon\left|\underset{(\mathbb{R}^{2})^{2}}{\int}u(t,x)\nabla V(x-y)\mathfrak{R}(t,x,y)dxdy\right|
\]

\[
\leq\epsilon||\nabla V\star\mathfrak{U}||_{L^{\infty}L^{\infty}}||u\omega||_{L^{\infty}L^{1}}+\epsilon||\nabla V\star\omega||_{L^{\infty}L^{\infty}}||u\mathfrak{U}||_{L^{\infty}L^{1}}
\]

\begin{onehalfspace}
\begin{equation}
+\epsilon^{2}||\nabla V\star\mathfrak{U}||_{L^{\infty}L^{\infty}}||u\mathfrak{U}||_{L^{\infty}L^{1}}\lesssim\epsilon.\label{eq:-33}
\end{equation}

Therefore by theorem (\ref{Serfaty })

\[
|J_{2}|
\]

\[
\leq\underset{(\mathbb{R}^{2})^{N}}{\int}\left|\underset{(\mathbb{R}^{2})^{2}-\triangle}{\int}(u(t,x)-u(t,y))\nabla V(x-y)(\mu_{X_{N}}-\omega)^{\otimes2}(x,y)dxdy\right|\rho_{N,\epsilon,\hbar}(t)dX_{N}+O(\epsilon)
\]

\[
\leq\underset{(\mathbb{R}^{2})^{N}}{\int}C||\nabla u(t,\cdot)||_{L^{\infty}}((\mathfrak{f}_{N}(X_{N},\omega(t,\cdot))+C(1+||\omega(t,\cdot)||_{L^{\infty}})N^{-\frac{1}{3}}+\frac{\log N}{N})\rho_{N,\epsilon,\hbar}(t)dX_{N}
\]

\end{onehalfspace}

\[
+\underset{(\mathbb{R}^{2})^{N}}{\int}2C||u(t,\cdot)||_{W^{1,\infty}}(1+||\omega(t,\cdot)||_{L^{\infty}})N^{-\frac{1}{2}}))\rho_{N,\epsilon,\hbar}(t)dX_{N}+O(\epsilon)
\]

\begin{onehalfspace}
\[
\leq\underset{(\mathbb{R}^{2})^{N}}{\int}C||\nabla u(t,\cdot)||_{L^{\infty}}\mathfrak{f}_{N}(X_{N},\omega(t,\cdot))\rho_{N,\epsilon,\hbar}(t)dX_{N}
\]

\end{onehalfspace}

\[
+C||\nabla u||_{L^{\infty}L^{\infty}}(1+||\omega||_{L^{\infty}L^{\infty}})\frac{1}{N^{\frac{1}{3}}}+C||\nabla u||_{L^{\infty}L^{\infty}}\frac{\log N}{N}
\]

\begin{onehalfspace}
\[
+2C||u||_{L^{\infty}W^{1,\infty}}(1+||\omega||_{L^{\infty}L^{\infty}})N^{-\frac{1}{2}}+O(\epsilon).
\]

By proposition (\ref{low bound }) we have 

\[
\underset{(\mathbb{R}^{2})^{N}}{\int}C||\nabla u(t,\cdot)||_{L^{\infty}}\mathfrak{f}_{N}(X_{N},\omega)\rho_{N,\epsilon,\hbar}(t)dX_{N}
\]

\end{onehalfspace}

\[
=\frac{1}{2}\underset{(\mathbb{R}^{2})^{N}}{\int}C||\nabla u(t,\cdot)||_{L^{\infty}}(\mathfrak{f}_{N}(X_{N},\omega)+\frac{(1+||\omega(t,\cdot)||_{L^{\infty}})}{N}+\frac{\log N}{N})\rho_{N,\epsilon,\hbar}(t)dX_{N}
\]

\begin{onehalfspace}
\[
-\frac{1}{2}C||\nabla u(t,\cdot)||_{L^{\infty}}(\frac{(1+||\omega(t,\cdot)||_{L^{\infty}})}{N}+\frac{\log N}{N})
\]

\[
\leq C||\nabla u||_{L^{\infty}L^{\infty}}(\mathcal{E}_{2}^{\ast}(t)+\frac{1+||\omega||_{L^{\infty}L^{\infty}}}{2N}+\frac{\log N}{N})
\]

\end{onehalfspace}

\[
+\frac{1}{2}C||\nabla u||_{L^{\infty}L^{\infty}}(\frac{1+||\omega||_{L^{\infty}L^{\infty}}}{2N}+\frac{\log N}{N}).
\]

\begin{onehalfspace}
To finish the estimate for $J_{2}$, we shall replace $\mathcal{E}_{2}^{\ast}$
by $\mathcal{E}_{2}$, with the expense of adding a term vanishing
as $\epsilon\rightarrow0$. Indeed, the bound below follows from the
remarks following the definition of $\mathcal{E}(t)$ 
\end{onehalfspace}

\begin{equation}
\epsilon||(V\star\mathfrak{U})\omega||_{L^{\infty}L^{1}}+\epsilon^{2}||(V\star\mathfrak{U})\mathfrak{U}||_{L^{\infty}L^{1}}+\epsilon||\rho_{N:1}(V\star\mathfrak{U})||_{L^{\infty}L^{1}}\lesssim\epsilon,\label{eq:-35}
\end{equation}

\begin{onehalfspace}
and shows that 
\begin{equation}
\mathcal{E}_{2}(t)+O(\epsilon)=\mathcal{E}_{2}^{\ast}(t).\label{eq:-34}
\end{equation}

Hence 

\[
|J_{2}|\leq C||\nabla u||_{L^{\infty}L^{\infty}}(\mathcal{E}_{2}(t)+\frac{1+||\omega||_{L^{\infty}L^{\infty}}}{2N}+\frac{\log N}{N})
\]

\end{onehalfspace}

\[
+\frac{1}{2}C||\nabla u||_{L^{\infty}L^{\infty}}(\frac{1+||\omega||_{L^{\infty}L^{\infty}}}{2N}+\frac{\log N}{N})
\]

\begin{onehalfspace}
\[
+C||\nabla u||_{L^{\infty}L^{\infty}}(1+||\omega||_{L^{\infty}L^{\infty}})\frac{1}{N^{\frac{1}{3}}}+C||\nabla u||_{L^{\infty}L^{\infty}}\frac{\log N}{N}
\]

\begin{equation}
+2C||u||_{L^{\infty}W^{1,\infty}}(1+||\omega||_{L^{\infty}L^{\infty}})N^{-\frac{1}{2}}+O(\epsilon).\label{eq:-11-3}
\end{equation}

\textit{Estimating $J_{3},J_{4}$ and $J_{5}$.}

The terms $J_{3},J_{4}$ and $J_{5}$ are of order $\epsilon$:

\begin{equation}
|J_{3}|\leq\epsilon||\nabla(-\Delta)^{-1}(u\mathfrak{U})(\rho_{N:1}(t)-\omega-\epsilon\mathfrak{U})||_{1}\label{eq:-10-1}
\end{equation}

\[
\leq\epsilon(||\nabla(-\Delta)^{-1}(u\mathfrak{U})||_{L^{\infty}L^{\infty}}+||\nabla(-\Delta)^{-1}(u\mathfrak{U})||_{L^{\infty}L^{\infty}}||\omega||_{L^{\infty}L^{1}})
\]

\end{onehalfspace}

\[
+\epsilon^{2}||\mathfrak{U}||_{L^{\infty}L^{1}}||\nabla(-\Delta)^{-1}(u\mathfrak{U})||_{L^{\infty}L^{\infty}})\lesssim\epsilon.
\]

\begin{onehalfspace}
\[
|J_{4}|\leq\epsilon||\partial_{t}p(\omega+\epsilon\mathfrak{U}-\rho_{N:1})||_{1}
\]

\[
\leq\epsilon(||\partial_{t}p||_{L^{\infty}L^{\infty}}||\rho_{N:1}||_{L^{\infty}L^{1}}+||\partial_{t}p||_{L^{\infty}L^{\infty}}||\omega||_{L^{\infty}L^{1}})
\]

\begin{equation}
+\epsilon^{2}||\partial_{t}p||_{L^{\infty}L^{\infty}}||\mathfrak{U}||_{L^{\infty}L^{1}}\lesssim\epsilon.\label{eq:-11-2}
\end{equation}

\end{onehalfspace}

Finally, not that 
\[
\frac{\epsilon}{2N}\stackrel[i=1]{N}{\sum}\mathrm{trace}(|x_{i}^{k}|^{2}R_{N}(t))\leq\mathcal{E}_{1}(t),
\]

thus 

\begin{equation}
|J_{5}|\leq\frac{\epsilon}{N}\stackrel[i=1]{N}{\sum}\mathrm{trace}((|x_{i}^{k}|^{2}+|u^{k}|^{2}(t,x_{i}))R_{N}(t))\leq2\mathcal{E}_{1}(t)+\epsilon||u||_{L^{\infty}L^{\infty}}^{2}.\label{eq:-32}
\end{equation}

\begin{onehalfspace}
By proposition (\ref{low bound }) and inequality (\ref{eq:-35})
there is a constant $\gamma$ such that 
\[
E(t)\coloneqq\mathcal{E}(t)+\frac{1+||\omega||_{L^{\infty}L^{\infty}}}{N}+\frac{\log N}{N}+\gamma\epsilon\geq0.
\]

By inequalities (\ref{eq:-15}),(\ref{eq:-11-3}), (\ref{eq:-10-1}),
(\ref{eq:-11-2}), (\ref{eq:-32}) and Gronwall, $E(t)\rightarrow0$
as $\frac{1}{N}+\epsilon+\hbar\rightarrow0$. This a fortiori implies
$\mathcal{E}(t)\rightarrow0$ as $\frac{1}{N}+\epsilon+\hbar\rightarrow0$,
which concludes the proof. 
\end{onehalfspace}
\begin{onehalfspace}
\begin{flushright}
$\square$
\par\end{flushright}
\end{onehalfspace}

\section{Weak Convergence}

Here we show that the asymptotic vanishing of $\mathcal{E}(t)$ implies
weak convergence $\rho_{N:1,\epsilon,\hbar}\rightarrow\omega$. This
fact is essentially already contained in the literature (see e.g.
{[}16{]}, section 3.4), and is included mainly for the sake of clarity
of exposition. In the expense of a bit more technique, one can show
that in fact $\rho_{N:1,\epsilon,\hbar}\rightarrow\omega$ in some
appropriate negative Sobolev space -- see {[}16{]}, section 3.4
for a guidance on how this is to be done. A key ingredient to establish
this weak convergence is the following coercivity estimate for the
modulated energy
\begin{prop}
\begin{flushleft}
\textup{({[}19{]}, Proposition 3.6) }\label{Coercivity} Let $0<\alpha\leq1$.
There are constants $C=C(\alpha)>0,\lambda=\lambda(\alpha)>0$ such
that for any $\varphi\in C^{0,\alpha}(\mathbb{R}^{2})$, any probability
density $\mu\in L^{\infty}(\mathbb{R}^{2})$ and any $X_{N}$ with
$\forall i\neq j:x_{i}\neq x_{j}$ it holds that 
\par\end{flushleft}
\begin{flushleft}
\[
\left|\underset{\mathbb{R}^{2}}{\int}\varphi(x)(\frac{1}{N}\stackrel[i=1]{N}{\sum}\delta_{x_{j}}-\mu)(dx)\right|
\]
\par\end{flushleft}
\[
\leq C||\varphi||_{C^{0,\alpha}(\mathbb{R}^{2})}N^{-\frac{\lambda}{2}}+C||\varphi||_{\dot{H}^{1}(\mathbb{R}^{2})}\left(\mathfrak{f}(X_{N},\mu)+\frac{1+||\mu||_{\infty}}{N}+\frac{\log N}{N}\right)^{\frac{1}{2}}.
\]
\end{prop}
Let $X_{N}$ with $\forall i\neq j:x_{i}\neq x_{j}$. Specifying proposition
(\ref{Coercivity}) to $\mu=\omega$ and $\alpha=1$, we see there
are constant $\lambda,C>0$ such that for any $\varphi\in C^{0,1}(\mathbb{R}^{2})$
we have 

\[
\left|\underset{\mathbb{R}^{2}}{\int}\varphi(x)(\mu_{N}-\omega)(x)dx\right|
\]

\begin{equation}
\leq C||\varphi||_{C^{0,1}(\mathbb{R}^{2})}N^{-\frac{\lambda}{2}}+C||\varphi||_{\dot{H}^{1}(\mathbb{R}^{2})}\left(\mathfrak{f}(X_{N},\omega)+\frac{(1+||\omega||_{L^{\infty}L^{\infty}})}{N}+\frac{\log N}{N}\right)^{\frac{1}{2}}.\label{eq:-24}
\end{equation}

In addition we have the following simple identity, observed in {[}18{]},
lemma 7.4. The proof is included for the sake of completeness 
\begin{lem}
\begin{flushleft}
\label{Estimate } If $\rho_{N}$ is a symmetric probability density
on $(\mathbb{R}^{2})^{N},$ $\mu\in L^{\infty}(\mathbb{R}^{2})\cap L^{1}(\mathbb{R}^{2})$
and $\varphi\in C_{b}(\mathbb{R}^{2})$ then\textbf{ }
\[
\underset{\mathbb{R}^{2}}{\int}\varphi(x)(\rho_{N:1}-\mu)(x)dx=\frac{1}{N}\underset{(\mathbb{R}^{2})^{N}}{\int}\left(\underset{\mathbb{R}^{2}}{\int}(\stackrel[i=1]{N}{\sum}\delta_{x_{i}}-N\mu)\varphi(x)dx\right)\rho_{N}(X_{N})dX_{N}.
\]
\par\end{flushleft}
\end{lem}
\begin{proof}
By symmetry of $\rho_{N}$ we have the identity 

\[
\underset{\mathbb{R}^{2}}{\int}\varphi(x_{1})\rho_{N:1}(x_{1})dx_{1}=\underset{(\mathbb{R}^{2})^{N}}{\int}\frac{1}{N}\stackrel[i=1]{N}{\sum}\varphi(x_{i})\rho_{N}(x_{1},...,x_{N})dx_{1}...dx_{N}.
\]

Therefore

\[
\underset{\mathbb{R}^{2}}{\int}\varphi(x)(\rho_{N:1}-\mu)(x)dx
\]

\[
=\underset{(\mathbb{R}^{2})^{N}}{\int}\frac{1}{N}\stackrel[i=1]{N}{\sum}\varphi(x_{i})\rho_{N}(x_{1},...,x_{N})dx_{1}...dx_{N}-\underset{\mathbb{R}^{2}}{\int}\varphi(x)\mu(x)dx
\]

\[
=\underset{(\mathbb{R}^{2})^{N}}{\int}\frac{1}{N}\stackrel[i=1]{N}{\sum}\varphi(x_{i})\rho_{N}(x_{1},...,x_{N})dx_{1}...dx_{N}
\]

\[
-\frac{1}{N}\underset{(\mathbb{R}^{2})^{N}}{\int}\left(\underset{\mathbb{R}^{2}}{\int}\varphi(x)N\mu(x)dx\right)\rho_{N}(x_{1},...,x_{N})dX_{N}
\]

\[
=\frac{1}{N}\underset{(\mathbb{R}^{2})^{N}}{\int}\rho_{N}\left(\underset{\mathbb{R}^{2}}{\int}(\stackrel[i=1]{N}{\sum}\delta_{x_{i}}-N\mu)\varphi dx\right)dX_{N}.
\]
\end{proof}
Integrating both sides of inequality (\ref{eq:-24}) against $\rho_{N,\epsilon,\hbar}(t,X_{N})$
and owing to lemma (\ref{Estimate }) gives 

\[
\left|\underset{\mathbb{R}^{2}}{\int}\varphi(x)(\rho_{N:1}-\omega)(x)dx\right|
\]

\[
\leq\underset{(\mathbb{R}^{2})^{N}}{\int}\left|\underset{\mathbb{R}^{2}}{\int}\varphi(x)(\frac{1}{N}\stackrel[i=1]{N}{\sum}\delta_{x_{j}}-\omega)(x)dx\right|\rho_{N,\epsilon,\hbar}(t,X_{N})dX_{N}
\]

\[
\leq C||\varphi||_{C^{0,1}(\mathbb{R}^{2})}N^{-\frac{\lambda}{2}}
\]

\[
+C||\varphi||_{\dot{H}^{1}(\mathbb{R}^{2})}\underset{(\mathbb{R}^{2})^{N}}{\int}\left(\mathfrak{f}(X_{N},\omega)+\frac{(1+||\omega||_{L^{\infty}C^{0}})}{N}+\frac{\log N}{N}\right)^{\frac{1}{2}}\rho_{N,\epsilon,\hbar}(t,X_{N})dX_{N}
\]

\begin{equation}
\leq C||\varphi||_{C^{0,1}(\mathbb{R}^{2})}N^{-\frac{\lambda}{2}}+C||\varphi||_{\dot{H}^{1}(\mathbb{R}^{2})}\left(\mathcal{E}_{2}^{\ast}(t)+\frac{(1+||\omega||_{L^{\infty}C^{0}})}{N}+\frac{\log N}{N}\right)^{\frac{1}{2}},\label{eq:-22-1}
\end{equation}

by Cauchy Schwartz. Since $\mathcal{E}_{2}(t)+O(\epsilon)=\mathcal{E}_{2}^{\ast}(t)$
(see (\ref{eq:-34})) inequality (\ref{eq:-22-1}) is recast as 

\begin{onehalfspace}
\[
\left|\underset{\mathbb{R}^{2}}{\int}\varphi(x)(\rho_{N:1}-\omega)(x)dx\right|
\]

\begin{equation}
\leq C||\varphi||_{C^{0,1}(\mathbb{R}^{2})}N^{-\frac{\lambda}{2}}+C||\varphi||_{\dot{H}^{1}(\mathbb{R}^{2})}(\mathcal{E}_{2}(t)+O(\epsilon)+\frac{(1+||\omega||_{L^{\infty}C^{0}})}{N}+\frac{\log N}{N})^{\frac{1}{2}}.\label{eq:-23-1}
\end{equation}

\end{onehalfspace}

Since $\mathcal{E}_{1}(t)\geq0,$ theorem (\ref{Main Thm}) together
with proposition (\ref{low bound }) evidently imply $\mathcal{E}_{2}(t)\rightarrow0$
as $\frac{1}{N}+\epsilon+\hbar\rightarrow0$. Hence, the right hand
side of (\ref{eq:-23-1}) goes to $0$ as $\frac{1}{N}+\epsilon+\hbar\rightarrow0$,
which shows that $\rho_{N:1}\rightarrow\omega$ weakly as $\frac{1}{N}+\epsilon+\hbar\rightarrow0$. 

\section{\label{sec:5 OF CHAP 2} Typicality of the Assumption $\mathcal{E}_{N,\epsilon,\hbar}(0)\protect\underset{\frac{1}{N}+\epsilon+\hbar}{\rightarrow}0$}

We explain here how to construct initial data $R_{N,\epsilon,\hbar}^{in}$
which witnesses the fact that the condition $\mathcal{E}(0)\rightarrow0$
is nonempty. It will be convenient to use Toeplitz operators for the
construction of $R_{N,\epsilon,\hbar}^{in}$. We borrow some definitions
and elementary facts from {[}5{]} regarding Toeplitz operators. Let
$z=(q,p)\in\mathbb{R}^{d}\times\mathbb{R}^{d}$. For each $\hbar$
we consider the complex valued function on $\mathbb{R}^{d}$ defined
by 
\[
\left|z,\hbar\right\rangle (x)\coloneqq(\pi\hbar)^{-\frac{d}{4}}e^{-\frac{|x-q|^{2}}{2\hbar}}e^{\frac{ip\cdot x}{\hbar}}.
\]
We denote by $\left|z,\hbar\right\rangle \left\langle z,\hbar\right|:\mathfrak{H}\rightarrow\mathfrak{H}$
the orthogonal projection on the line$\left|z,\hbar\right\rangle \mathbb{C}$
in $L^{2}(\mathbb{R}^{d}\times\mathbb{R}^{d})$. For each finite/positive
Borel measure $\nu$ on $\mathbb{R}^{d}\times\mathbb{R}^{d}$ we define
the Toeplitz operator $\mathrm{OP}_{\hbar}^{T}(\nu):\mathfrak{H}\rightarrow\mathfrak{H}$
at scale $\hbar$ with symbol $\nu$ by the formula 

\[
\mathrm{OP}_{\hbar}^{T}(\nu)=\frac{1}{(2\pi\hbar)^{d}}\underset{\mathbb{R}^{d}\times\mathbb{R}^{d}}{\int}\left|z,\hbar\right\rangle \left\langle z,\hbar\right|\nu(dz).
\]

If $\nu$ is a positive measure 
\begin{equation}
\mathrm{OP}_{\hbar}^{T}(\nu)=\mathrm{OP}_{\hbar}^{T}(\nu)^{\ast}\geq0\label{eq:-16-1}
\end{equation}
 and

\begin{equation}
\mathrm{trace}(\mathrm{OP}_{\hbar}^{T}(\nu))=\frac{1}{(2\pi\hbar)^{d}}\underset{\mathbb{R}^{d}\times\mathbb{R}^{d}}{\int}\nu(dz).\label{eq:-17-1}
\end{equation}

In addition, if $\nu$ is a probability measure then $\mathrm{OP}_{\hbar}^{T}(\nu)$
is a trace class operator. We also recall the definition of the Wigner
and Husimi transforms. 
\begin{defn}
\begin{flushleft}
\label{husimi transform definition } Let $A$ be an unbounded operator
on $L^{2}(\mathbb{R}^{d})$ with integral kernel $k_{A}\in\mathcal{S}'(\mathbb{R}^{d}\times\mathbb{R}^{d})$.
The \textit{Wigner transform of scale $\hbar$ }of $A$ is the distribution
on $\mathbb{R}^{d}\times\mathbb{R}^{d}$ defined by the formula
\[
W_{\hbar}[A]\coloneqq(2\pi)^{-d}\mathcal{F}_{2}(k_{A}\circ j_{\hbar}),
\]
$\hbar$
\par\end{flushleft}
where $j_{\hbar}(x,y)\coloneqq(x+\frac{1}{2}\hbar y,x-\frac{1}{2}\hbar y)$
and $\mathcal{F}_{2}$ is the partial Fourier transform with respect
to the second variable.
\begin{flushleft}
The \textit{Husimi transform of scale $\hbar$ }is the function on
$\mathbb{R}^{d}\times\mathbb{R}^{d}$ defined by the formula 
\[
\widetilde{W_{\hbar}}[A](x,\xi)\coloneqq e^{\frac{\hbar\Delta_{(x,\xi)}}{4}}W_{\hbar}[R](x,\xi).
\]
\par\end{flushleft}
\end{defn}
\begin{rem}
\begin{flushleft}
Denoting by $G_{a}^{d}$ the centered Gaussian density on $\mathbb{R}^{d}$
with covariance matrix $aI$ we can equivalently write 
\par\end{flushleft}
\begin{flushleft}
\[
\widetilde{W_{\hbar}}[R](x,\xi)=(G_{\frac{\hbar}{2}}^{2d}\star W_{\hbar})[R].
\]
\par\end{flushleft}
\begin{flushleft}
If $\mu$ is a finite/positive Borel probability measure on $\mathbb{R}^{d}\times\mathbb{R}^{d}$
then 
\[
W_{\hbar}[\mathrm{OP}_{\hbar}^{T}(\mu)]=\frac{1}{(2\pi\hbar)^{d}}G_{\frac{\hbar}{2}}^{2d}\star\mu.
\]
\par\end{flushleft}
\begin{flushleft}
(see formula (51) in {[}5{]}). 
\par\end{flushleft}

\end{rem}
We gather a few elementary formulas in the following 
\begin{thm}
\begin{flushleft}
\textup{\label{Toeplitz formulas } ({[}5{]}, formulas (46) and (48))}
1. $\mathrm{OP}_{\hbar}^{T}(1)=I_{\mathfrak{H}}$.
\par\end{flushleft}
\begin{flushleft}
2. If $f$ is a quadratic form on $\mathbb{R}^{d}$ then 
\par\end{flushleft}

\end{thm}
\[
\mathrm{OP}_{\hbar}^{T}(f(q)dqdp)=f(x)+\frac{1}{4}\hbar(\Delta f)I_{\mathfrak{H}},
\]

\[
\mathrm{OP}_{\hbar}^{T}(f(p)dqdp)=f(-i\hbar\partial_{x})+\frac{1}{4}\hbar(\Delta f)I_{\mathfrak{H}}.
\]

\textit{3. If $\nu$ is a finite/positive Borel measure on $\mathbb{R}^{d}\times\mathbb{R}^{d}$
then }

\[
\mathrm{trace}(\mathrm{OP}_{\hbar}^{T}(\nu)A)=\underset{\mathbb{R}^{d}\times\mathbb{R}^{d}}{\int}\widetilde{W}_{\hbar}[A](z)\nu(dz).
\]

We will need the following 
\begin{lem}
\label{formula} Let $\nu\in\mathcal{P}(\mathbb{R}^{2}\times\mathbb{R}^{2})$
have finite second moments. Then 
\[
\mathrm{trace}((-i\hbar\partial_{x^{k}}+\frac{1}{2\epsilon}(x^{\bot})^{k})^{2}\mathrm{OP}_{\hbar}^{T}((2\pi\hbar)^{2}\nu)))
\]

\[
=\underset{\mathbb{R}^{2}\times\mathbb{R}^{2}}{\int}\left|p+\frac{1}{2\epsilon}q^{\bot}\right|^{2}\nu(dqdp)+\frac{\hbar}{4\epsilon^{2}}+\frac{1}{2}d\hbar.
\]
\end{lem}
\textit{Proof.} \textbf{Calculation of the cross terms}. Put 
\[
A=-i\hbar\partial_{x^{k}}\lor\frac{1}{2\epsilon}(x^{\bot})^{k},
\]

and denote by $k_{A}(x,y)$ the kernel of $A$. It is readily checked
that 

\[
k_{A}(x,y)=i\hbar(\frac{1}{\epsilon}x^{\bot})^{k}\delta'_{k}(y-x),
\]

which implies 
\[
W_{\hbar}[A](x,\xi)=(2\pi)^{-2}\mathcal{F}_{2}(k_{A}\circ j_{\hbar})=\frac{i\hbar}{\epsilon}(2\pi)^{-2}\mathcal{F}_{2}((x+\frac{1}{2}\hbar y)^{\bot k}\delta'_{k}(-\hbar y))
\]

\[
=\frac{(2\pi\hbar)^{-2}}{\epsilon}\xi_{k}(x^{\bot})^{k},
\]

hence 

\[
\widetilde{W_{\hbar}}[A](q,p)=G_{\frac{\hbar}{2}}^{2\times2}\star_{x,\xi}W_{\hbar}[A](q,p)=\frac{(2\pi\hbar)^{-2}}{\epsilon}p\cdot q^{\bot},
\]

where the last equality is due to the formula 

\begin{equation}
G_{\frac{\hbar}{2}}^{2d}\star_{x,\xi}g=\underset{0\leq n\leq\frac{m}{2}}{\sum}\frac{\hbar^{n}}{4^{n}n!}\Delta_{x,\xi}^{n}g\label{eq:-18-1}
\end{equation}

for any polynomial $g$ of degree $\leq m$. 

By 3. of theorem (\ref{Toeplitz formulas }), it follows that
\begin{equation}
\mathrm{trace}(-i\hbar\partial_{x^{k}}\lor\frac{1}{2\epsilon}(x^{\bot})^{k}\mathrm{OP}_{\hbar}^{T}((2\pi\hbar)^{2}\nu))=\underset{\mathbb{R}^{2}}{\int}\frac{1}{\epsilon}p\cdot q^{\bot}\nu(dqdp).\label{eq:-19-1}
\end{equation}

\textbf{Calculation of the dominant terms}. We utilize theorem (\ref{Toeplitz formulas })
with $f(p)=|p|^{2}$. 

\[
\mathrm{trace}(-\hbar^{2}\Delta_{x_{k}}\mathrm{OP}_{\hbar}^{T}((2\pi\hbar)^{2}\nu)))=\mathrm{trace}((\mathrm{OP}_{\hbar}^{T}(f(p))-\frac{1}{4}\hbar(\Delta f)I_{\mathfrak{H}})\mathrm{OP}_{\hbar}^{T}((2\pi\hbar)^{2}\nu)))
\]

\[
=\mathrm{trace}(\mathrm{OP}_{\hbar}^{T}(f(p))\mathrm{OP}_{\hbar}^{T}((2\pi\hbar)^{2}\nu)))-\hbar
\]

\[
=(2\pi\hbar)^{2}\underset{\mathbb{R}^{2}\times\mathbb{R}^{2}}{\int}\widetilde{W_{\hbar}}[\mathrm{OP}_{\hbar}^{T}(f)](q,p)\nu(dqdp)-\hbar.
\]

Owing to formula (\ref{eq:-18-1}) we find 

\[
(2\pi\hbar)^{2}\underset{\mathbb{R}^{2}\times\mathbb{R}^{2}}{\int}\widetilde{W_{\hbar}}[\mathrm{OP}_{\hbar}^{T}(f)](q,p)\nu(dqdp)
\]

\[
=(2\pi\hbar)^{2}\underset{\mathbb{R}^{2}\times\mathbb{R}^{2}}{\int}G_{\frac{\hbar}{2}}^{2\times2}\star_{x,\xi}W_{\hbar}[\mathrm{OP}_{\hbar}^{T}(f)](q,p)\nu(dqdp)
\]

\[
=\underset{\mathbb{R}^{2}\times\mathbb{R}^{2}}{\int}(G_{\frac{\hbar}{2}}^{2\times2}\star G_{\frac{\hbar}{2}}^{2\times2}\star f)(p)\nu(dqdp)=\underset{\mathbb{R}^{2}\times\mathbb{R}^{2}}{\int}|p|^{2}\nu(dqdp)+2\hbar,
\]

so that 
\begin{equation}
\mathrm{trace}(-\hbar^{2}\Delta_{x_{k}}\mathrm{OP}_{\hbar}^{T}((2\pi\hbar)^{2}\nu))=\underset{\mathbb{R}^{2}\times\mathbb{R}^{2}}{\int}|p|^{2}\nu(dqdp)+\hbar.\label{eq:-21-1}
\end{equation}

Next, we utilize theorem (\ref{Toeplitz formulas }) with $f(q)=|q|^{2}$.
\[
\mathrm{trace}(\frac{1}{4\epsilon^{2}}|x|^{2}\mathrm{OP}_{\hbar}^{T}((2\pi\hbar)^{2}\nu))
\]
\[
=\frac{1}{4\epsilon^{2}}\mathrm{trace}((\mathrm{OP}_{\hbar}^{T}(f(q))-\hbar I_{\mathfrak{H}})\mathrm{OP}_{\hbar}^{T}((2\pi\hbar)^{2}\nu)))
\]

\[
=\frac{1}{4\epsilon^{2}}\mathrm{trace}((\mathrm{OP}_{\hbar}^{T}(f(q))\mathrm{OP}_{\hbar}^{T}((2\pi\hbar)^{2}\nu)))-\frac{\hbar}{4\epsilon^{2}},
\]

and once again by formula (\ref{eq:-18-1})

\[
\frac{1}{4\epsilon^{2}}\mathrm{trace}(\mathrm{OP}_{\hbar}^{T}(f(q))\mathrm{OP}_{\hbar}^{T}((2\pi\hbar)^{2}\nu)))
\]

\[
=\frac{(2\pi\hbar)^{2}}{4\epsilon^{2}}\underset{\mathbb{R}^{2}\times\mathbb{R}^{2}}{\int}\widetilde{W_{\hbar}}[\mathrm{OP}_{\hbar}^{T}(f)](q,p)\nu(dqdp)
\]

\[
=\frac{1}{4\epsilon^{2}}\underset{\mathbb{R}^{2}\times\mathbb{R}^{2}}{\int}(G_{\frac{\hbar}{2}}^{2\times2}\star G_{\frac{\hbar}{2}}^{2\times2}\star f)(q,p)\nu(dqdp)=\frac{1}{4\epsilon^{2}}\underset{\mathbb{R}^{2}\times\mathbb{R}^{2}}{\int}|q|^{2}\nu(dqdp)+\frac{1\hbar}{2\epsilon^{2}}.
\]

so that 

\begin{equation}
\frac{1}{4\epsilon^{2}}\mathrm{trace}(\mathrm{OP}_{\hbar}^{T}(f(q))\mathrm{OP}_{\hbar}^{T}((2\pi\hbar)^{d}\nu))=\frac{1}{4\epsilon^{2}}\underset{\mathbb{R}^{2}\times\mathbb{R}^{2}}{\int}|q|^{2}\nu(dqdp)+\frac{\hbar}{4\epsilon^{2}}.\label{eq:-20-1}
\end{equation}

Gathering (\ref{eq:-19-1}),(\ref{eq:-21-1}) and (\ref{eq:-20-1})
produces the asserted identity. 
\begin{flushright}
$\square$
\par\end{flushright}
\begin{rem}
Considerations similar to the ones demonstrated in the proof of lemma
\ref{formula} show that the regularity condition in assumption (A)
is indeed satisfied. 
\end{rem}
We can now choose $\hbar=\hbar(\epsilon)$ such that $\frac{\hbar}{\epsilon}\rightarrow0$.
Under this assumption we can exemplify the assumption $\mathcal{E}(0)\rightarrow0$.
We proceed through the following steps 

\textbf{Step 1.} \textbf{The kinetic part. }Pick $\theta\in L^{1}(\mathbb{R}^{2})\cap L^{\infty}(\mathbb{R}^{2})$
and take 
\[
\nu\coloneqq\nu_{\epsilon}(dqdp)=\omega(q)\delta(p+\frac{1}{2\epsilon}q^{\bot}-\theta(q)),
\]
 where $\omega$ is a $C_{0}^{\infty}(\mathbb{R}^{2})$ probability
density, and let 
\[
R_{N,\epsilon,\hbar}^{in}=\mathrm{OP}_{\hbar}^{T}((2\pi\hbar)^{2}\nu)^{\otimes N}.
\]
 It is evident from (\ref{eq:-16-1})-(\ref{eq:-17-1}) that this
is indeed a density matrix. Moreover by lemma (\ref{formula}) we
have 

\[
\epsilon\mathrm{trace}(((-i\hbar\partial_{x^{k}}+\frac{1}{2\epsilon}(x^{\bot})^{k})^{2}\mathrm{OP}_{\hbar}^{T}((2\pi\hbar)^{2}\nu))=\epsilon\underset{\mathbb{R}^{2}\times\mathbb{R}^{2}}{\int}\left|p+\frac{1}{2\epsilon}q^{\bot}\right|^{2}\nu(dqdp)+\frac{\hbar}{4\epsilon}+\epsilon\hbar.
\]

Clearly, by assumption 

\[
\frac{\hbar}{4\epsilon}+\epsilon\hbar\rightarrow0.
\]

We have 

\[
\underset{\mathbb{R}^{2}\times\mathbb{R}^{2}}{\int}\left|p+\frac{1}{2\epsilon}q^{\bot}\right|^{2}\nu(dqdp)=\underset{\mathbb{R}^{2}}{\int}\omega(q)|\theta(q)|^{2}dq\leq||\varpi||_{1}||\theta||_{\infty}.
\]

In addition 

\[
|\mathrm{trace}(|u^{0}|^{2}\mathrm{OP}_{\hbar}^{T}((2\pi\hbar)^{2}\nu))|\leq||u^{0}||_{\infty}^{2},
\]

and 

\[
|\mathrm{trace}(u^{0}\lor((-i\hbar\partial_{x^{k}}+\frac{1}{2\epsilon}(x^{\bot})^{k}))\mathrm{OP}_{\hbar}^{T}((2\pi\hbar)^{2}\nu))|
\]

\[
\leq\mathrm{trace}(|u^{0}\lor(-i\hbar\partial_{x^{k}}+\frac{1}{2\epsilon}(x^{\bot})^{k})|\mathrm{OP}_{\hbar}^{T}((2\pi\hbar)^{2}\nu))
\]

\[
\leq\mathrm{trace}((|u^{0}|^{2}+|(-i\hbar\partial_{x^{k}}+\frac{1}{2\epsilon}(x^{\bot})^{k})|^{2})\mathrm{OP}_{\hbar}^{T}((2\pi\hbar)^{2}\nu))).
\]

The above inequalities evidently entail

\[
\epsilon\mathrm{trace}(((-i\hbar\partial_{x^{k}}+\frac{1}{2\epsilon}(x^{\bot})^{k}-u^{0})^{2}\mathrm{OP}_{\hbar}^{T}((2\pi\hbar)^{2}\nu))\underset{\epsilon+\hbar\rightarrow0}{\rightarrow}0.
\]

As for the quadratic term, by formulas (52) and (53) in {[}5{]} (or
equivalently by a similar calculation to the one demonstrated in lemma
(\ref{formula})) 

\[
\epsilon\mathrm{trace}(|x|^{2}\mathrm{OP}_{\hbar}^{T}((2\pi\hbar)^{d}\nu))
\]

\[
=\epsilon(1+\hbar)\underset{\mathbb{R}^{2}\times\mathbb{R}^{2}}{\int}|q|^{2}\nu(dqdp)=\epsilon(1+\hbar)\underset{\mathbb{R}^{2}}{\int}|q|^{2}\omega(q)dq\underset{\epsilon+\hbar}{\rightarrow}0.
\]

This takes care of the kinetic part. 

\textbf{Step 2.} \textbf{The interaction part. }Note that if $\rho_{N}$
is given as a pure tensor product $\rho_{N}=\rho^{\otimes N}$ then 

\[
\underset{\mathbb{R}^{2}\times\mathbb{R}^{2}}{\int}V(x-y)(\frac{N-1}{N}\rho_{N:2}(x,y)+\mu(x)\mu(y)-2\rho_{N:1}(x)\mu(y))dxdy
\]

\[
=\underset{\mathbb{R}^{2}}{\int}(\frac{N-1}{N}V\star\rho-V\star\mu)(x)\rho(x)dx+\underset{\mathbb{R}^{2}}{\int}V\star\mu(x)(\mu-\rho)(x)dx\coloneqq I+J.
\]

Let $\rho_{\epsilon,\hbar}(x)$ denote the integral kernel of $\mathrm{OP}_{\hbar}^{T}((2\pi\hbar)^{2}\nu_{\epsilon})$
evaluated at the diagonal $(x,x)$. By a straightforward calculation 

\[
\rho_{\epsilon,\hbar}(x)=\rho_{\epsilon,\hbar}(x)=\frac{1}{\pi\hbar}\underset{\mathbb{R}^{2}}{\int}e^{-\frac{|x-q|^{2}}{\hbar}}\omega(q)dq=(G_{\frac{\hbar}{2}}^{2}\star\omega)(x).
\]

\textbf{Substep 2.1.} $J\underset{\hbar+\epsilon\rightarrow0}{\rightarrow}0$.
We claim that $J\underset{\hbar+\epsilon\rightarrow0}{\rightarrow}0$.
Assuming that $\omega$ has compact support we have 
\[
|J|=\left|\underset{\mathbb{R}^{2}}{\int}V\star(\omega+\epsilon\mathfrak{U})(x)(\omega+\epsilon\mathfrak{U}-\rho)(x)dx\right|
\]

\[
\leq\left|\underset{\mathbb{R}^{2}}{\int}\omega(x)(G_{\frac{\hbar}{2}}^{2}\star V\star\omega-V\star\omega(x))(x)dx\right|
\]

\[
+\epsilon\left|\underset{\mathbb{R}^{2}}{\int}\omega(x)V\star\mathfrak{U}(x)dx\right|+\epsilon\left|\underset{\mathbb{R}^{2}}{\int}V\star\mathfrak{U}(x)(\omega+\epsilon\mathfrak{U}-G_{\frac{\hbar}{2}}^{2}\star\omega)(x)dx\right|
\]

\[
\leq\left|\underset{\mathbb{R}^{2}}{\int}\omega(x)(G_{\frac{\hbar}{2}}^{2}\star V\star\omega-V\star\omega(x))(x)dx\right|
\]

\[
+\epsilon(2||V\star\mathfrak{U}||_{\infty}+\epsilon||V\star\mathfrak{U}||_{\infty}||\mathfrak{U}||_{1}+||V\star\mathfrak{U}||_{\infty}).
\]

Clearly the terms in the second line are of order $\epsilon$. The
term in the first line is handled as follows. Decompose $V$ into
its negative and positive parts $V=V^{-}+V^{+}$. We have that 
\[
\left|\underset{\mathbb{R}^{2}}{\int}\omega(x)(G_{\frac{\hbar}{2}}^{2}\star V\star\omega-V\star\omega)(x)dx\right|
\]

\[
=\left|\underset{\mathbb{R}^{2}}{\int}V\star\omega(x)(G_{\frac{\hbar}{2}}^{2}\star\omega-\omega(x))(x)dx\right|
\]

\[
\leq\left|\underset{\mathbb{R}^{2}}{\int}V^{-}\star\omega(x)(G_{\frac{\hbar}{2}}^{2}\star\omega-\omega)(x)dx\right|
\]

\[
+\left|\underset{\mathbb{R}^{2}}{\int}V^{+}\star\omega(x)(G_{\frac{\hbar}{2}}^{2}\star\omega-\omega)(x)dx\right|.
\]

Since $V^{+}\in L^{p}(\mathbb{R}^{2})$ for all $1\leq p<\infty$,
the same is true for the convolution $V^{+}\star\omega$, and so since
$G_{\frac{\hbar}{2}}^{2}$ is an approximation to the identity we
get 

\[
\left|\underset{\mathbb{R}^{2}}{\int}V^{+}\star\omega(x)(G_{\frac{\hbar}{2}}^{2}\star\omega-\omega)dx\right|\leq||V^{+}\star\omega||_{2}\left\Vert G_{\frac{\hbar}{2}}^{2}\star\omega-\omega\right\Vert _{2}\underset{\hbar\rightarrow0}{\rightarrow}0.
\]

In addition, observing that $|V^{-}|(x)\leq|x|^{2}$, one finds 

\[
\left|\underset{\mathbb{R}^{2}}{\int}V^{-}\star\omega(x)(G_{\frac{\hbar}{2}}^{2}\star\omega-\omega)dx\right|
\]

\[
\leq\underset{\mathbb{R}^{2}}{\int}|V^{-}\star\omega|(x)|G_{\frac{\hbar}{2}}^{2}\star\omega-\omega|dx
\]

\[
\leq\underset{\mathbb{R}^{2}}{\int}|\cdot|^{2}\star\omega(x)|G_{\frac{\hbar}{2}}^{2}\star\omega-\omega|dx.
\]

Now pick $\omega(x)=\frac{1}{\Lambda}\chi G_{\frac{1}{2}}^{2}(x)$
where $\chi\in C_{0}^{\infty}(\mathbb{R}^{2})$ with $0\leq\chi\leq1$
and $\Lambda\coloneqq||\chi G_{\frac{1}{2}}^{2}||_{1}$. We split
the last integral as 

\[
\underset{\mathbb{R}^{2}}{\int}|\cdot|^{2}\star\omega(x)|G_{\frac{\hbar}{2}}^{2}\star\omega-\omega|dx
\]

\[
=\underset{|x|\leq\hbar^{-\beta}}{\int}|\cdot|^{2}\star\omega(x)|G_{\frac{\hbar}{2}}^{2}\star\omega-\omega|dx+\underset{|x|>\hbar^{-\beta}}{\int}|\cdot|^{2}\star\omega(x)|G_{\frac{\hbar}{2}}^{2}\star\omega-\omega|dx.
\]

Owing to formula (\ref{eq:-18-1}), the first integral is mastered
as follows 

\[
\underset{|x|\leq\hbar^{-\beta}}{\int}|\cdot|^{2}\star\omega(x)|G_{\frac{\hbar}{2}}^{2}\star\omega-\omega|dx\leq\frac{1}{\Lambda}\underset{|x|\leq\hbar^{-\beta}}{\int}|\cdot|^{2}\star G_{\frac{1}{2}}^{2}(x)|G_{\frac{\hbar}{2}}^{2}\star\omega-\omega|dx
\]

\[
=\frac{1}{\Lambda}\underset{|x|\leq\hbar^{-\beta}}{\int}(|x|^{2}+1)|G_{\frac{\hbar}{2}}^{2}\star\omega-\omega|dx\leq\frac{1}{\Lambda}\left(\underset{|x|\leq\hbar^{-\beta}}{\int}(|x|^{2}+1)^{2}dx\right)^{\frac{1}{2}}||G_{\frac{\hbar}{2}}^{2}\star\omega-\omega||_{L^{2}(\mathbb{R}^{2})}
\]

\begin{equation}
\leq\frac{1}{\Lambda}\left(\frac{(\hbar^{-2\beta}+1)^{3}-1}{6}\right)^{\frac{1}{2}}||G_{\frac{\hbar}{2}}^{2}\star\omega-\omega||_{L^{2}(\mathbb{R}^{2})}.\label{eq:integral near origin}
\end{equation}

By Plancheral the right hand side of equation (\ref{eq:integral near origin})
is

\[
\leq\frac{2}{\Lambda}\hbar^{-3\beta}||\widehat{G_{\frac{\hbar}{2}}^{2}}\widehat{\omega}-\widehat{\omega}||_{2}=\frac{2}{\Lambda}\hbar^{-3\beta}||e^{\frac{-\hbar|\xi|^{2}}{4}}\widehat{\omega}-\widehat{\omega}||_{2}
\]

\[
\leq\frac{2}{\Lambda}\hbar^{-3\beta}\left\Vert \frac{\hbar|\xi|^{2}}{4}\widehat{\omega}\right\Vert _{2}=\frac{1}{2\Lambda}\hbar^{1-3\beta}\left\Vert |\xi|^{2}\widehat{\omega}\right\Vert _{2},
\]

where the last inequality is thanks to the elementary inequality $1-e^{-r}\leq r$
for $r\in[0,\infty)$. Taking $0<\beta<\frac{1}{3}$ yields 

\[
\underset{\mathbb{R}^{2}}{\int}|\cdot|^{2}\star\omega(x)|G_{\frac{\hbar}{2}}^{2}\star\omega-\omega|dx\underset{\hbar\rightarrow0}{\rightarrow}0.
\]

Finally, the formula for convolution of Gaussians gives 
\[
G_{\frac{\hbar}{2}}^{2}\star G_{\frac{1}{2}}^{2}=G_{\frac{\hbar+1}{2}}^{2}
\]

so that

\[
\underset{|x|>\hbar^{-\beta}}{\int}|\cdot|^{2}\star\omega(x)|G_{\frac{\hbar}{2}}^{2}\star\omega-\omega|dx
\]

\[
\leq\underset{|x|>\hbar^{-\beta}}{\int}|\cdot|^{2}\star\omega(x)(G_{\frac{\hbar}{2}}^{2}\star\omega+\omega)dx\leq\underset{|x|>\hbar^{-\beta}}{\int}|\cdot|^{2}\star\omega(x)(\frac{1}{\Lambda}G_{\frac{\hbar}{2}}^{2}\star G_{\frac{1}{2}}^{2}+\omega)dx
\]

\[
=\underset{|x|>\hbar^{-\beta}}{\int}|\cdot|^{2}\star\omega(x)(\frac{1}{\Lambda}G_{\frac{\hbar+1}{2}}^{2}+\omega)dx
\]

\[
\leq\frac{1}{\Lambda}\underset{|x|>\hbar^{-\beta}}{\int}(1+|x|^{2})(\frac{1}{\Lambda}G_{1}^{2}+\omega)dx\underset{\hbar\rightarrow0}{\rightarrow}0
\]

as a tail of a converging integral. 

\textbf{Substep 2.2.} We utilize the formula $G_{\frac{\hbar}{2}}^{2}\star G_{\frac{\hbar}{2}}^{2}=G_{\hbar}^{2}$
to find 

\[
|I|=\left|\underset{\mathbb{R}^{2}}{\int}(\frac{N-1}{N}G_{\frac{\hbar}{2}}^{2}\star(V\star G_{\frac{\hbar}{2}}^{2}\star\omega)-G_{\frac{\hbar}{2}}^{2}\star(V\star\mu))(x)\omega(x)dx\right|
\]

\[
\leq\left|\underset{\mathbb{R}^{2}}{\int}(\frac{N-1}{N}G_{\hbar}^{2}\star(V\star\omega)-G_{\frac{\hbar}{2}}^{2}\star(V\star\omega))\omega(x)dx\right|
\]

\[
\text{+}\epsilon\left|\underset{\mathbb{R}^{2}}{\int}(V\star\mathfrak{U})(x)\rho(x)dx\right|
\]

\[
\leq\frac{N-1}{N}\left|\underset{\mathbb{R}^{2}}{\int}(G_{\hbar}^{2}\star(V\star\omega)-G_{\frac{\hbar}{2}}^{2}\star(V\star\omega))\omega(x)dx\right|
\]

\[
+\frac{1}{N}\left|\underset{\mathbb{R}^{2}}{\int}G_{\frac{\hbar}{2}}^{2}\star(V\star\omega)\omega(x)dx\right|+\epsilon||V\star\mathfrak{U}||_{\infty}.
\]

The first term is recognized as a Cauchy difference (by substep 2.1)
and therefore vanishes as $\hbar\rightarrow0$, while the second and
third terms are of order $\frac{1}{N}$ and $\epsilon$ respectively,
because $\left|\underset{\mathbb{R}^{2}}{\int}G_{\frac{\hbar}{2}}^{2}\star(V\star\omega)\omega(x)dx\right|$
is uniformly bounded in $\hbar$ (again by substep 2.1).

\section{Essential self-adjointness of the Hamiltonian. Conservation of energy}

We elaborate on the self-adjointness of the quantum Hamiltonian and
give a proof of lemma (\ref{ENERGY CON}). Recall the notation $A(x)\coloneqq\frac{1}{2\epsilon}x^{\bot}$
and the operators defined by the formulas 
\[
\mathscr{K}_{N}\coloneqq\frac{1}{2}\stackrel[j=1]{N}{\sum}(-i\hbar\nabla_{x_{j}}+A(x_{j}))^{2}+\frac{1}{2}\stackrel[j=1]{N}{\sum}|x_{j}|^{2}\coloneqq K_{N}+\frac{1}{2}\stackrel[j=1]{N}{\sum}|x_{j}|^{2},
\]

\begin{equation}
\mathscr{V}_{N}\coloneqq\frac{1}{N\epsilon}\underset{p<q}{\sum}V_{pq}\label{N BODY MULTI OPERATOR}
\end{equation}

on $C_{0}^{\infty}(\mathbb{R}^{2N})$. For each fixed $1\leq p\leq N$
and $k=1,2$ denote by $\Pi^{p,k}$ the closure of the operator $-i\hbar\partial_{x_{p}^{k}}+A^{k}(x_{p})$
in $C_{0}^{\infty}(\mathbb{R}^{2N})$ (it is closable as a symmetric
operator). With slight abuse of notation, denote by $K_{N}$ the operator 

\[
K_{N}\coloneqq\frac{1}{2}\stackrel[p=1]{N}{\sum}\underset{k}{\sum}(\Pi^{p,k})^{\ast}\Pi^{p,k},
\]

which is self-adjoint according to theorem X.25 in {[}15{]}. Equivalently\footnote{This equivalence is shown in {[}24{]} }
we may view $K_{N}$ (and $\mathscr{K}_{N}$) as unbounded self-adjoint
operators through the use of the terminology of quadratic forms, which
we now briefly review following section 2.3 in {[}23{]}. Given a
non-negative quadratic form $q:\mathfrak{Q}\rightarrow\mathbb{C}$
($\mathfrak{Q}$ is a dense subspace of the underlying Hilbert space
$\mathfrak{H}$) corresponding to a sesquilinear form $s:\mathfrak{Q}\times\mathfrak{Q}\rightarrow\mathbb{C}$
define a scalar product $\left\langle \cdot,\cdot\right\rangle _{q}:\mathfrak{Q}\times\mathfrak{Q}\rightarrow\mathbb{C}$
by letting 

\[
\left\langle \psi,\varphi\right\rangle _{q}\coloneqq s(\psi,\varphi)+\left\langle \psi,\varphi\right\rangle .
\]

The norm inherited from this scalar product is denoted $||\cdot||_{q}$
(i.e. $||\varphi||_{q}^{2}\coloneqq q(\varphi)+||\varphi||^{2}$).
If $q$ is closable, the completion of $\mathfrak{Q}$ with respect
to $||\cdot||_{q}$ is a subspace of $\mathfrak{H}$ denoted $\mathfrak{H}_{q}$. 
\begin{thm}
Let $q$ be a non-negative closed quadratic form on $\mathfrak{Q}$
corresponding to a sesquilinear form $s$. There is a unique self-adjoint
operator $T$ with form domain $\mathfrak{Q}(T)$ such that $\mathfrak{Q}=\mathfrak{Q}(T)$
and $q=q_{T}$. In addition 

\[
D(T)=\{\psi\in\mathfrak{H}_{q}|\exists\widetilde{\psi}\in\mathfrak{H}\forall\varphi\in\mathfrak{H}_{q}:s(\varphi,\psi)=\left\langle \varphi,\widetilde{\psi}\right\rangle \}
\]

\[
T\psi=\widetilde{\psi}.
\]
\end{thm}
Consider the sesquilinear form 
\[
S(\varphi_{N},\psi_{N})\coloneqq\frac{1}{2}\stackrel[j=1]{N}{\sum}\underset{\mathbb{R}^{2N}}{\int}\overline{(-i\hbar\nabla_{x_{j}}+A(x_{j}))\varphi_{N}}(-i\hbar\nabla_{x_{j}}+A(x_{j}))\psi_{N}dX_{N}
\]
on $C_{0}^{\infty}(\mathbb{R}^{2N})\times C_{0}^{\infty}(\mathbb{R}^{2N})$
whose corresponding quadratic form is $Q$. The completion of $C_{0}^{\infty}(\mathbb{R}^{2N})$
with respect to $||\cdot||_{Q}$ is denoted $\mathfrak{Q}$ and the
extension of $Q$ to $\mathfrak{Q}$ is denoted $q$. With slight
abuse of notation, denote by $K_{N}$ the self-adjoint operator arising
via the above theorem with $\mathfrak{Q}$ and $q$ as defined in
the previous line. Similarly (with slight abuse of notation) $\mathscr{K}_{N}$
is the self-adjoint operator arising from the form 
\[
\mathscr{S}(\varphi_{N},\psi_{N})\coloneqq\frac{1}{2}\stackrel[j=1]{N}{\sum}\underset{\mathbb{R}^{2N}}{\int}\overline{(-i\hbar\nabla_{x_{j}}+A(x_{j}))\varphi_{N}}(-i\hbar\nabla_{x_{j}}+A(x_{j}))\psi_{N}dX_{N}+\frac{1}{2}\stackrel[j=1]{N}{\sum}\underset{\mathbb{R}^{2N}}{\int}\overline{\varphi_{N}}|x_{j}|^{2}\psi_{N}dX_{N}
\]

with $\mathfrak{Q}$ being the completion of $C_{0}^{\infty}(\mathbb{R}^{2N})$
with respect to $||\cdot||_{\mathscr{Q}}$, where $\mathscr{Q}$ is
the quadratic form corresponding to $\mathscr{S}$. Since both $K_{N}$
and$\mathscr{K}_{N}$ (as operators on $C_{0}^{\infty}(\mathbb{R}^{2N})$)
are known to be essentially self-adjoint (theorem (1.1) in {[}21{]}),
the above self-adjoint extensions are in fact unique. The Kato-Rellich
perturbation theory for self-adjoint operators is the most fundamental
tool used for the purpose of viewing the quantum Hamiltonian, which
is a perturbation of $\mathscr{K}_{N}$, as an unbounded self-adjoint
operator. 
\begin{thm}
\label{Kato's theorem } Let $T,D(T)\subset\mathfrak{H}$ be a (essentially)
self-adjoint operator and $S,D(S)\subset\mathfrak{H}$ a symmetric
operator such that $D(T)\subset D(S)$. Suppose there exist $0<a<1,b>0$
such that for each $\varphi\in D(T)$ 
\begin{equation}
||S\varphi||^{2}\leq a||T\varphi||^{2}+b||\varphi||^{2}.\label{eq:-25}
\end{equation}
Then $T+S,D(T+S)=D(T)$ is (essentially) self-adjoint. In the case
where $T$ is essentially self-adjoint we have $D(\overline{T})\subset D(\overline{S})$
and $\overline{T+S}=\overline{T}+\overline{S}$. 
\end{thm}
If $T,S$ satisfy the condition (\ref{eq:-25}) for some $a>0$ and
$b>0$, we say that $S$ is $T$-bounded. The infimum over all $a>0$
for which there exist a $b>0$ such that (\ref{eq:-25}) holds is
called the relative bound of $S$. The theorem stated below, which
is one of the central conclusions from Kato-Rellich theory, shows
that the self-adjointness of the $N$-body Laplacian is invariant
under a perturbation by a symmetric $N$-body potential satisfying
a certain integrability condition. 
\begin{thm}
\label{relative bound for Kato potentials } Let $V\in L^{\infty}(\mathbb{R}^{2})+L^{2}(\mathbb{R}^{2})$
be a real valued function. Let $\mathscr{V}_{N}$ be the multiplication
operator defined by the formula (\ref{N BODY MULTI OPERATOR}) with
domain $D(\mathscr{V}_{N})=H^{2}(\mathbb{R}^{2N})$. Let 

\[
-\Delta_{N,\hbar}\coloneqq-\frac{\hbar^{2}}{2}\stackrel[j=1]{N}{\sum}\Delta_{x_{j}}
\]
 viewed as a self-adjoint operator with domain $H^{2}(\mathbb{R}^{2N})$.
Then $\mathscr{V}_{N}$ is $-\Delta_{N,\hbar}$-bounded with relative
bound $0$. Consequently, $-\Delta_{N,\hbar}+\mathscr{V}_{N}$ is
self-adjoint on $H^{2}(\mathbb{R}^{2N})$. 
\end{thm}
Although $\log(x)$ fails to satisfy the assumption stated in (\ref{relative bound for Kato potentials }),
the negative part of $\log(x)$ evidently does verify this assumption,
and this observation will be important in the argument that we employ.
The work of Avron-Herbst-Simon {[}1{]} extends theorem \ref{relative bound for Kato potentials }
for the case where a magnetic field is included. In particular they
prove the following result 
\begin{thm}
\textup{\label{AVRON SIMON } (Theorem 2.4 in {[}1{]}){]}}. Let $\mathscr{V}_{N}$
be a multiplication operator on $L^{2}(\mathbb{R}^{2N})$. Suppose
that $\mathscr{V}_{N}$ is $-\Delta_{N}$-bounded with relative bound
$\alpha$. Then $\mathscr{V}_{N}$ is $K_{N}$-bounded with relative
bound $\leq\alpha$.
\end{thm}
In contrast to {[}1{]}, the Hamiltonian of interest contains an attractive
potential of the form $\frac{1}{2}\stackrel[j=1]{N}{\sum}|x_{j}|^{2}$.
It is the aim of the forthcoming lemmata to show that the essential
self-adjointness of the Hamiltonian in this case is obtained as a
corollary of the result reported in theorem (\ref{AVRON SIMON }).
As we will see, the reason for this is that a bound on $\mathscr{V}_{N}$
with respect to $K_{N}$ is sharper than a bound with respect to $\mathscr{K}_{N}$.
The first step is 
\begin{lem}
\label{estimate on moments and gradient } For each $\varphi\in C_{0}^{\infty}(\mathbb{R}^{2})$
one has the following estimates 
\[
\underset{\mathbb{R}^{2}}{\int}|x|^{2}|\varphi|^{2}(x)dx\leq||\mathscr{K}\varphi||_{2}^{2}+||\varphi||_{2}^{2}
\]

and 

\[
\underset{\mathbb{R}^{2}}{\int}\hbar^{2}|\nabla\varphi|^{2}(x)dx\leq\max(2,\frac{1}{2\epsilon^{2}})(||\mathscr{K}\varphi||_{2}^{2}+||\varphi||_{2}^{2}).
\]
\end{lem}
\textit{Proof. }With the abbreviation $\mathscr{K}=\mathscr{K}_{1}$,
$K=K_{1}$ and $\Pi=(\Pi^{1,1},\Pi^{1,2})$, one has 

\[
\underset{\mathbb{R}^{2}}{\int}|x|^{2}|\varphi|^{2}(x)dx\leq\left\langle K\varphi,\varphi\right\rangle +\underset{\mathbb{R}^{2}}{\int}|x|^{2}|\varphi|^{2}(x)dx
\]

\[
=\underset{\mathbb{R}^{2}}{\int}\overline{\varphi}(x)(K+|x|^{2})\varphi(x)dx
\]

\begin{equation}
\leq2||\mathscr{K}\varphi||_{2}||\varphi||_{2}\leq||\mathscr{K}\varphi||_{2}^{2}+||\varphi||_{2}^{2},\label{eq:-29-1}
\end{equation}

which is the first inequality. The second inequality is implied from
(\ref{eq:-29-1}) as follows 

\[
\underset{\mathbb{R}^{2}}{\int}\hbar^{2}|\nabla\varphi|^{2}(x)dx=||(\Pi-A(x))\varphi||_{2}^{2}
\]

\[
\leq2||\Pi\varphi||_{2}^{2}+\frac{1}{2\epsilon^{2}}\underset{\mathbb{R}^{2}}{\int}|x|^{2}|\varphi|^{2}(x)dx
\]

\[
=4\underset{\mathbb{R}^{2}}{\int}\overline{\varphi}(x)K\varphi(x)dx+\frac{1}{2\epsilon^{2}}\underset{\mathbb{R}^{2}}{\int}|x|^{2}|\varphi|^{2}(x)dx
\]

\[
\leq\max(4,\frac{1}{\epsilon^{2}})\underset{\mathbb{R}^{2}}{\int}\overline{\varphi}(x)(K+\frac{1}{2}|x|^{2})\varphi(x)dx
\]

\[
\leq\max(2,\frac{1}{2\epsilon^{2}})(||\mathscr{K}\varphi||_{2}^{2}+||\varphi||_{2}^{2}).
\]

\begin{flushright}
$\square$
\par\end{flushright}

The next lemma shows that the norm associated with $\mathscr{K}_{N}$
controls the norm associated with $K_{N}$ up to a constant 
\begin{lem}
\label{K^2+|x|^2>K^2} There is a constant $C=C(\hbar,\epsilon,N)$
such that for all $\varphi_{N}\in C_{0}^{\infty}(\mathbb{R}^{2N})$
it holds that 

\[
||\mathscr{K}_{N}\varphi_{N}||_{2}^{2}+||\varphi_{N}||_{2}^{2}\geq C(||K_{N}\varphi_{N}||_{2}^{2}+||\varphi_{N}||_{2}^{2}).
\]
\end{lem}
\begin{proof}
Let us first consider the case $N=1$ (with the abbreviation $\mathscr{K}=\mathscr{K}_{1}$,
$K=K_{1}$ and $\Pi^{1,k}=\Pi^{k}$). For each $\varphi\in C_{0}^{\infty}(\mathbb{R}^{2})$
we expand 

\[
||\mathscr{K}\varphi||_{2}^{2}
\]

\begin{equation}
=||K\varphi||_{2}^{2}+\Re\left\langle K\varphi,|x|^{2}\varphi\right\rangle +\frac{1}{4}\left\Vert |x|^{2}\varphi\right\Vert _{2}^{2}.\label{eq:-2}
\end{equation}

The second term is estimated from below as follows (applying Einstein's
summation for $k,l$)

\[
\Re\left\langle K\varphi,|x|^{2}\varphi\right\rangle 
\]

\[
=\epsilon^{2}\Re\left\langle K\varphi,(\Pi^{k}+\overline{\Pi^{k}})^{2}\varphi\right\rangle 
\]

\[
=\epsilon^{2}\Re\left\langle K\varphi,((\Pi^{k})^{2}+\Pi^{k}\vee\overline{\Pi^{k}}+(\overline{\Pi^{k}})^{2})\varphi\right\rangle 
\]

\[
=\epsilon^{2}\Re\left\langle K\varphi,(\Pi^{k})^{2}\varphi\right\rangle +\epsilon^{2}\Re\left\langle K\varphi,\Pi^{k}\vee\overline{\Pi^{k}}\varphi\right\rangle 
\]

\[
+\epsilon^{2}\Re\left\langle K\varphi,(\overline{\Pi^{k}})^{2}\varphi\right\rangle 
\]

\[
=\epsilon^{2}\Re\left\langle K\varphi,(\Pi^{k})^{2}\varphi\right\rangle +2\epsilon^{2}\Re\left\langle K\varphi,\Pi^{k}\overline{\Pi^{k}}\varphi\right\rangle 
\]

\[
+\epsilon^{2}\Re\left\langle K\varphi,(\overline{\Pi^{k}})^{2}\varphi\right\rangle 
\]

\[
=\frac{\epsilon^{2}}{2}\Re\left\langle (\Pi^{l})^{2}\varphi,(\Pi^{k})^{2}\varphi\right\rangle +\epsilon^{2}\Re\left\langle (\Pi^{l})^{2}\varphi,\Pi^{k}\overline{\Pi^{k}}\varphi\right\rangle 
\]

\begin{equation}
+\frac{\epsilon^{2}}{2}||\Pi^{l}\overline{\Pi^{k}}\varphi||_{2}^{2},\label{eq:-1}
\end{equation}

where we used the simple observation that $\Pi^{l}$ and $\overline{\Pi^{k}}$
commute. In addition, for each $l,k$ the following relation is easily
observed 

\[
\Pi^{l}\Pi^{k}-\Pi^{k}\Pi^{l}=\frac{1}{2\epsilon}i\hbar\mathbf{J}_{kl}
\]

where we recall

\[
\mathbf{J}\coloneqq2\begin{pmatrix}0 & 1\\
-1 & 0
\end{pmatrix}.
\]

Therefore 
\[
\left\langle (\Pi^{l})^{2}\varphi,(\Pi^{k})^{2}\varphi\right\rangle =\left\langle \Pi^{l}\varphi,\Pi^{l}\Pi^{k}\Pi^{k}\varphi\right\rangle 
\]

\[
=\left\langle \Pi^{l}\varphi,(\Pi^{k}\Pi^{l}+\frac{i\hbar}{2\epsilon}\mathbf{J}_{kl})\Pi^{k}\varphi\right\rangle 
\]

\[
=\left\langle \Pi^{k}\Pi^{l}\varphi,\Pi^{l}\Pi^{k}\varphi\right\rangle +\left\langle \Pi^{l}\varphi,\frac{i\hbar}{2\epsilon}\mathbf{J}_{kl}\Pi^{k}\varphi\right\rangle 
\]

\[
=\left\langle \Pi^{k}\Pi^{l}\varphi,(\Pi^{k}\Pi^{l}+\frac{i\hbar}{2\epsilon}\mathbf{J}_{kl})\varphi\right\rangle +\left\langle \Pi^{l}\varphi,\frac{i\hbar}{2\epsilon}\mathbf{J}_{kl}\Pi^{k}\varphi\right\rangle 
\]

\begin{equation}
=||\Pi^{k}\Pi^{l}\varphi||_{2}^{2}+\left\langle \Pi^{l}\varphi,\frac{i\hbar}{\epsilon}\mathbf{J}_{kl}\Pi^{k}\varphi\right\rangle .\label{eq:-3}
\end{equation}

Moreover 

\[
\Re\left\langle (\Pi^{l})^{2}\varphi,\Pi^{k}\overline{\Pi^{k}}\varphi\right\rangle =\Re\left\langle \Pi^{l}\varphi,\Pi^{l}\Pi^{k}\overline{\Pi^{k}}\varphi\right\rangle 
\]

\[
=\Re\left\langle \overline{\Pi^{k}}\Pi^{l}\varphi,\Pi^{l}\Pi^{k}\varphi\right\rangle 
\]

\begin{equation}
=\Re\left\langle \Pi^{l}\overline{\Pi^{k}}\varphi,\Pi^{l}\Pi^{k}\varphi\right\rangle \geq-\frac{1}{2}(||\Pi^{l}\overline{\Pi^{k}}\varphi||_{2}^{2}+||\Pi^{l}\Pi^{k}\varphi||_{2}^{2}).\label{eq:-4}
\end{equation}

To conclude, equations (\ref{eq:-1}), (\ref{eq:-3}) and inequality
(\ref{eq:-4}) entail

\[
\Re\left\langle K\varphi,|x|^{2}\varphi\right\rangle 
\]

\[
\geq\frac{\epsilon^{2}}{2}\Re\left\langle \Pi^{l}\varphi,\frac{i\hbar}{\epsilon}\mathbf{J}_{kl}\Pi^{k}\varphi\right\rangle 
\]

\[
+\frac{\epsilon^{2}}{2}||\Pi^{l}\Pi^{k}\varphi||_{2}^{2}-\frac{\epsilon^{2}}{2}||\Pi^{l}\overline{\Pi^{k}}\varphi||_{2}^{2}-\frac{\epsilon^{2}}{2}||\Pi^{l}\Pi^{k}\varphi||_{2}^{2}+\frac{\epsilon^{2}}{2}||\Pi^{l}\overline{\Pi^{k}}\varphi||_{2}^{2}
\]

\[
=\frac{\epsilon^{2}}{2}\Re\left\langle \Pi^{l}\varphi,\frac{i\hbar}{\epsilon}\mathbf{J}_{kl}\Pi^{k}\varphi\right\rangle ,
\]

hence

\[
\Re\left\langle K\varphi,|x|^{2}\varphi\right\rangle \geq\frac{\epsilon^{2}}{2}\Re\left\langle \Pi^{l}\varphi,\frac{i\hbar}{\epsilon}\mathbf{J}_{kl}\Pi^{k}\varphi\right\rangle .
\]

Integration by parts shows that 
\[
\frac{\epsilon^{2}}{2}\Re\left\langle \Pi^{l}\varphi,\frac{i\hbar}{\epsilon}\mathbf{J}_{kl}\Pi^{k}\varphi\right\rangle =\Re\left(i\hbar\epsilon\underset{\mathbb{R}^{2}}{\int}\overline{\varphi}A^{\bot}(x)\cdot i\hbar\nabla\varphi(x)dx\right)
\]

\[
\geq-\hbar^{2}\left(\frac{1}{4}\underset{\mathbb{R}^{2}}{\int}|x|^{2}|\varphi|^{2}+\frac{1}{2}\underset{\mathbb{R}^{2}}{\int}|\nabla\varphi|^{2}(x)dx\right),
\]

so that in light of lemma (\ref{estimate on moments and gradient })

\[
\Re\left\langle K\varphi,|x|^{2}\varphi\right\rangle 
\]

\[
\geq-\frac{\hbar^{2}}{2}(||\mathscr{K}\varphi||_{2}^{2}+||\varphi||_{2}^{2})-\max(2,\frac{1}{2\epsilon^{2}})(||\mathscr{K}\varphi||_{2}^{2}+||\varphi||_{2}^{2})
\]

\[
=-\alpha(\hbar,\epsilon)(||\mathscr{K}\varphi||_{2}^{2}+||\varphi||_{2}^{2}),
\]

where 

\[
\alpha(\hbar,\epsilon)\coloneqq(\frac{\hbar^{2}}{2}+\max(2,\frac{1}{2\epsilon^{2}})).
\]

Therefore identity (\ref{eq:-2}) implies 

\[
||\mathscr{K}\varphi||_{2}^{2}\geq||K\varphi||_{2}^{2}+\Re\left\langle K\varphi,|x|^{2}\varphi\right\rangle 
\]

\[
\geq||K\varphi||_{2}^{2}-\alpha(\hbar,\epsilon)(||\mathscr{K}\varphi||_{2}^{2}+||\varphi||_{2}^{2})
\]

that is 

\[
(1+\alpha(\hbar,\epsilon))||\mathscr{K}\varphi||_{2}^{2}+(1+\alpha(\hbar,\epsilon))||\varphi||_{2}^{2}\geq||K\varphi||_{2}^{2}+||\varphi||_{2}^{2}
\]

which is the same as 
\begin{equation}
||\mathscr{K}\varphi||_{2}^{2}+||\varphi||_{2}^{2}\geq C(\hbar,\epsilon)(||K\varphi||_{2}^{2}+||\varphi||_{2}^{2}),\label{smooth inequality}
\end{equation}

where 

\[
C(\hbar,\epsilon)\coloneqq\frac{1}{1+\alpha(\hbar,\epsilon)}.
\]

The case of arbitrary $N\geq1$ is deduced from the case $N=1$. Indeed,
putting 

\[
K^{p}\coloneqq\frac{1}{2}\underset{k}{\sum}(\Pi^{p,k})^{\ast}\Pi^{p,k}
\]

for each $\varphi_{N}\in C_{0}^{\infty}(\mathbb{R}^{2N})$ we have 

\[
||\mathscr{K}_{N}\varphi_{N}||_{2}^{2}+N||\varphi_{N}||_{2}^{2}
\]

\[
=\left\langle \left(\stackrel[p=1]{N}{\sum}(K^{p}+\frac{1}{2}|x_{p}|^{2})\right)\varphi_{N},\left(\stackrel[p=1]{N}{\sum}(K^{p}+\frac{1}{2}|x_{p}|^{2})\right)\varphi_{N}\right\rangle +N||\varphi_{N}||_{2}^{2}
\]

\[
=\stackrel[p=1]{N}{\sum}||(K^{p}+\frac{1}{2}|x_{p}|^{2})\varphi_{N}||_{2}^{2}+\underset{p\neq q}{\sum}\left\langle (K^{p}+\frac{1}{2}|x_{p}|^{2})\varphi_{N},(K^{q}+\frac{1}{2}|x_{q}|^{2})\varphi_{N}\right\rangle +N||\varphi_{N}||_{2}^{2}
\]

\begin{equation}
\geq\stackrel[p=1]{N}{\sum}||(K^{p}+\frac{1}{2}|x_{p}|^{2})\varphi_{N}||_{2}^{2}+||\varphi_{N}||_{2}^{2}\label{eq:}
\end{equation}

because for any $p\neq q$ we have that 

\[
\left\langle (K^{p}+\frac{1}{2}|x_{p}|^{2})\varphi_{N},(K^{q}+\frac{1}{2}|x_{q}|^{2})\varphi_{N}\right\rangle \geq0.
\]

So in view of inequality (\ref{smooth inequality}) we have that the
right hand side of (\ref{eq:}) is 

\[
\geq\stackrel[p=1]{N}{\sum}C(\hbar,\epsilon)(||K^{p}\varphi_{N}||_{2}^{2}+||\varphi_{N}||_{2}^{2})
\]

\[
\geq\frac{C(\hbar,\epsilon)}{N}\left\Vert \stackrel[p=1]{N}{\sum}K^{p}\varphi_{N}\right\Vert _{2}^{2}+NC(\hbar,\epsilon)||\varphi_{N}||_{2}^{2}
\]

\[
=\frac{C(\hbar,\epsilon)}{N}||K_{N}\varphi_{N}||_{2}^{2}+NC(\hbar,\epsilon)||\varphi_{N}||_{2}^{2},
\]

hence 
\begin{equation}
||\mathscr{K}_{N}\varphi_{N}||_{2}^{2}+||\varphi_{N}||_{2}^{2}\geq C(\hbar,\epsilon,N)(||K_{N}\varphi_{N}||_{2}^{2}+||\varphi_{N}||_{2}^{2}),\label{N smooth inequality}
\end{equation}

where 

\[
C(\hbar,\epsilon,N)\coloneqq\frac{1}{N^{2}(1+\alpha(\hbar,\epsilon))}.
\]
\end{proof}
By theorem (\ref{Kato's theorem }), in order to assert that $\mathscr{H}_{N},D(\mathscr{H}_{N})=C_{0}^{\infty}(\mathbb{R}^{2N})$
is essentialy self-adjoint it is sufficient to show that $\mathscr{V}_{N}$
, viewed as a multiplication operator with domain $D(\mathscr{V}_{N})=C_{0}^{\infty}(\mathbb{R}^{2N})$,
is relatively bounded (with relative bound $<1$) with respect to
$\mathscr{K}_{N}$. This is achieved in the following 
\begin{lem}
\label{RELATIVE BOUND } For each $0<a<1$ there is some $b=b(N,\epsilon,\hbar)>0$
such that for all $\varphi_{N}\in C_{0}^{\infty}(\mathbb{R}^{2N})$ 

\[
||\mathscr{V}_{N}\varphi_{N}||_{2}^{2}\leq a||\mathscr{K}_{N}\varphi_{N}||_{2}^{2}+b||\varphi_{N}||_{2}^{2}.
\]

In particular $\mathscr{H}_{N}=\mathscr{K}_{N}+\mathscr{V}_{N}$ is
essentially self-adjoint on $C_{0}^{\infty}(\mathbb{R}^{2N})$. 
\end{lem}
\textit{Proof.} Set $V^{+}\coloneqq\mathbf{1}_{|x|\geq1}\log(|x|)$
and $V^{-}\coloneqq\mathbf{1}_{|x|\leq1}\log(|x|)$. 

\textbf{Step 1.} \textbf{Relative bound for $\frac{1}{N\epsilon}\underset{p<q}{\sum}V_{pq}^{+}$.
}For each $\varphi_{N}\in C_{0}^{\infty}(\mathbb{R}^{2N})$ one has 

\[
\underset{\mathbb{R}^{2N}}{\int}|V^{+}|^{2}(|x_{p}-x_{q}|)|\varphi_{N}|^{2}(X_{N})dX_{N}
\]

\[
\leq2\underset{\mathbb{R}^{2N}}{\int}(|x_{p}|^{2}+|x_{q}|^{2})|\varphi_{N}|^{2}(X_{N})dX_{N}
\]

\[
\leq2\left\langle \varphi_{N},K^{p}\varphi_{N}\right\rangle +2\underset{\mathbb{R}^{2N}}{\int}|x_{p}|^{2}|\varphi_{N}|^{2}(X_{N})dX_{N}
\]

\[
+2\left\langle \varphi_{N},K^{q}\varphi_{N}\right\rangle +2\underset{\mathbb{R}^{2N}}{\int}|x_{q}|^{2}|\varphi|^{2}(X_{N})dX_{N}
\]

\[
=2\underset{\mathbb{R}^{2N}}{\int}\overline{\varphi_{N}}(X_{N})(K^{p}+K^{q}+|x_{p}|^{2}+|x_{q}|^{2})\varphi_{N}(X_{N})dX_{N}
\]

\[
\leq4\underset{\mathbb{R}^{2N}}{\int}\overline{\varphi_{N}}(X_{N})\mathscr{K}_{N}\varphi_{N}(X_{N})dX_{N}
\]

\begin{equation}
\leq4||\varphi_{N}||_{2}||\mathscr{K}_{N}\varphi_{N}||_{2}\leq2||\mathscr{K}_{N}\varphi_{N}||_{2}^{2}+2||\varphi_{N}||_{2}^{2}<\infty.\label{eq:-29-1-1}
\end{equation}

Therefore, given $a>0$ pick $\eta=\frac{a\epsilon^{2}}{16(N-1)}$
and $b_{1}=\frac{4(N-1)}{\epsilon^{2}\eta}$ in order to find 

\[
\left\Vert \frac{1}{N\epsilon}\underset{p<q}{\sum}V_{pq}^{+}\varphi\right\Vert _{2}^{2}\leq\frac{1}{N^{2}\epsilon^{2}}\left(\underset{p<q}{\sum}||V_{pq}^{+}\varphi||_{2}\right)^{2}
\]

\[
\leq\frac{N-1}{2N\epsilon^{2}}\underset{p<q}{\sum}||V_{pq}^{+}\varphi||_{2}^{2}
\]

\[
\leq\frac{N-1}{N\epsilon^{2}}\underset{\mathbb{R}^{2N}}{\int}\overline{\varphi}(X_{N})(2N\stackrel[p=1]{N}{\sum}K^{p}+2N|X_{N}|^{2})\varphi(X_{N})dX_{N}
\]

\[
\leq\frac{4(N-1)}{\epsilon^{2}}||\varphi||_{2}||\mathscr{K}_{N}\varphi||_{2}
\]

\[
\leq\frac{4\eta(N-1)}{\epsilon^{2}}||\mathscr{K}_{N}\varphi||_{2}^{2}+\frac{4(N-1)}{\epsilon^{2}\eta}||\varphi||_{2}^{2}
\]

\begin{equation}
=\frac{1}{4}a||\mathscr{K}_{N}\varphi||_{2}^{2}+b_{1}||\varphi||_{2}^{2}.\label{Bound on positive part}
\end{equation}

\textbf{Step 2. Relative bound for $\frac{1}{N\epsilon}\underset{p<q}{\sum}V_{pq}^{-}$}.
The combination of theorems (\ref{relative bound for Kato potentials })
and (\ref{AVRON SIMON }) with lemma (\ref{K^2+|x|^2>K^2}) implies
that for each $0<a<1$ there is some $b_{2}=b_{2}(\epsilon,\hbar,N)$
such that for all $\varphi_{N}\in C_{0}^{\infty}(\mathbb{R}^{2N})$
it holds that 

\[
\left\Vert \frac{1}{N\epsilon}\underset{p<q}{\sum}V_{pq}^{-}\varphi_{N}\right\Vert _{2}^{2}\leq\frac{1}{4}aC(\hbar,\epsilon,N)||K_{N}\varphi_{N}||_{2}^{2}+b_{2}||\varphi_{N}||_{2}^{2}
\]

\begin{equation}
\leq\frac{1}{4}a(||\mathscr{K}_{N}\varphi_{N}||_{2}^{2}+||\varphi_{N}||^{2})+b_{2}||\varphi_{N}||_{2}^{2}\leq\frac{1}{4}a||\mathscr{K}_{N}\varphi_{N}||_{2}^{2}+(b_{2}+1)||\varphi_{N}||_{2}^{2}\label{bound on the negative part}
\end{equation}

Therefore the combination inequalities (\ref{Bound on positive part})
and (\ref{bound on the negative part}) entails 
\[
\left\Vert \frac{1}{N\epsilon}\underset{p<q}{\sum}V_{pq}\varphi\right\Vert _{2}^{2}\leq a||\mathscr{K}_{N}\varphi_{N}||_{2}^{2}+b||\varphi_{N}||_{2}^{2}
\]

with $b=2(b_{1}+b_{2}+1)$. 
\begin{flushright}
$\square$
\par\end{flushright}

The following elementary observation will be used to prove lemma (\ref{ENERGY CON}) 
\begin{fact}
\textup{\label{footnote } (Footnote 3 in {[}6{]})} If $S=S^{\ast}\geq0$
is unbounded with domain $D(S)\subset\mathfrak{H}$ and $T=T^{\ast}\geq0$
is trace class with eigenvectors $(e_{m})_{m\geq1}\subset D(S)$ and
eigenvalues $(\lambda_{m})_{m\geq1}$ respectively, then $TST\geq0$
and 

\[
\mathrm{trace}(TST)=\stackrel[m=1]{\infty}{\sum}\lambda_{m}^{2}\left\langle \psi_{m},S\psi_{m}\right\rangle .
\]
\end{fact}
\textit{Proof of lemma (\ref{ENERGY CON})}. \textbf{Step 1. $\mathcal{\mathscr{H}}_{N}R_{N}(t)^{\frac{1}{2}}$
is Hilbert--Schmidt.} Let $\{e_{m}\}_{m\geq1}$ be a complete system
of eigenfunctions of $R_{N}^{in}$ with $R_{N}^{in}e_{m}=\lambda_{m}e_{m}$.
By fact (\ref{footnote })

\begin{equation}
||\mathcal{\mathscr{H}}_{N}\sqrt{R_{N}(t)}||_{`2}^{2}=\underset{m\geq1}{\sum}\lambda_{m}||\mathcal{\mathscr{H}}_{N}e_{m}||_{2}^{2}.\label{eq:-26}
\end{equation}

Owing to lemma (\ref{RELATIVE BOUND }) we get 

\[
||\mathcal{\mathscr{H}}_{N}e_{m}||_{2}^{2}=\left\langle (\mathscr{K}_{N}+\mathscr{V}_{N})e_{m},(\mathscr{K}_{N}+\mathscr{V}_{N})e_{m}\right\rangle 
\]

\[
\leq2||\mathscr{K}_{N}e_{m}||_{2}^{2}+2||\mathscr{V}_{N}e_{m}||_{2}^{2}\leq C_{N,\epsilon,\hbar}(||\mathscr{K}_{N}e_{m}||_{2}^{2}+||e_{m}||_{2}^{2}),
\]

which implies that the right hand side of (\ref{eq:-26}) is 

\[
\leq C_{N,\epsilon,\hbar}\underset{m\geq1}{\sum}\lambda_{m}||(I+\mathscr{K}_{N})e_{m}||_{\mathfrak{H}_{N}}^{2}=C_{N,\epsilon,\hbar}\mathrm{trace}_{\mathfrak{H}_{N}}(\sqrt{R_{N}^{in}}(I+\mathscr{K}_{N})^{2}\sqrt{R_{N}^{in}})<\infty
\]

by assumption. 

\textbf{Step 2.} \textbf{$\mathrm{trace}_{\mathfrak{H}_{N}}(\sqrt{R_{N}(t)}\mathcal{\mathscr{H}}_{N}\sqrt{R_{N}(t)})$
is constant.} As a result of step 1 and Cauchy-Schwarz we see that
for all $t$ the operator $\sqrt{R_{N}(t)}\mathcal{\mathscr{H}}_{N}\sqrt{R_{N}(t)}$
is trace class. Furthermore 

\[
\mathrm{trace}_{\mathfrak{H}_{N}}(\sqrt{R_{N}(t)}\mathcal{\mathscr{H}}_{N}\sqrt{R_{N}(t)})
\]

\[
=\stackrel[k=1]{\infty}{\sum}\lambda_{k}\left\langle \mathcal{U}_{N}(t)e_{k},\mathcal{\mathscr{H}}_{N}\mathcal{U}_{N}(t)e_{k}\right\rangle 
\]

\[
=\stackrel[k=1]{\infty}{\sum}\lambda_{k}\left\langle \mathcal{U}_{N}(t)e_{k},\mathcal{U}_{N}(t)\mathcal{\mathscr{H}}_{N}e_{k}\right\rangle 
\]

\begin{equation}
=\mathrm{\mathrm{trace}}(\sqrt{R_{N}^{in}}\mathcal{\mathscr{H}}_{N}\sqrt{R_{N}^{in}}).\label{eq:-27}
\end{equation}

\textbf{Step 3.} \textbf{Energy Conservation. }Denote by $S_{N}(t,X_{N},Y_{N})$
the integral kernel of $\sqrt{R_{N}(t)}$ . Then we have 

\[
\mathrm{trace}_{\mathfrak{H}_{N}}(\sqrt{R_{N}(t)}\mathcal{\mathscr{H}}_{N}\sqrt{R_{N}(t)})
\]

\[
=\stackrel[j=1]{N}{\sum}\underset{\mathbb{R}^{2N}\times\mathbb{R}^{2N}}{\int}\overline{S_{N}(t,X_{N},Y_{N})}(K^{j}+\frac{1}{2}|x_{j}|^{2})S_{N}(t,X_{N},Y_{N})dX_{N}dY_{N}
\]

\[
+\frac{1}{N\epsilon}\underset{1\leq j<k\leq N}{\sum}\underset{\mathbb{R}^{2N}\times\mathbb{R}^{2N}}{\int}\overline{S_{N}(t,X_{N},Y_{N})}V(x_{j}-x_{k})S_{N}(t,X_{N},Y_{N})dX_{N}dY_{N}
\]

\[
=\frac{1}{2}\stackrel[j=1]{N}{\sum}\underset{\mathbb{R}^{2N}\times\mathbb{R}^{2N}}{\int}|(-i\hbar\nabla_{x_{j}}+\frac{1}{2\epsilon}x_{j}^{\bot})S_{N}(t,X_{N},Y_{N})|^{2}+|x_{j}|^{2}|S_{N}(t,X_{N},Y_{N})|^{2}dX_{N}dY_{N}
\]

\[
+\frac{1}{N\epsilon}\underset{1\leq j<k\leq N}{\sum}\underset{\mathbb{R}^{2N}\times\mathbb{R}^{2N}}{\int}\overline{S_{N}(t,X_{N},Y_{N})}V(x_{j}-x_{k})S_{N}(t,X_{N},Y_{N})dX_{N}dY_{N}
\]

\[
=\frac{1}{2}\mathrm{trace}_{\mathfrak{H}_{N}}(\sqrt{R_{N}(t)}\mathscr{K}_{N}\sqrt{R_{N}(t)})+\frac{N-1}{2\epsilon}\underset{\mathbb{R}^{2N}}{\int}V_{12}\rho_{N}(t,X_{N})dX_{N}.
\]

The claim follows from equation (\ref{eq:-27}). 
\begin{onehalfspace}
\begin{flushright}
$\square$
\par\end{flushright}
\end{onehalfspace}

\section{Appendix A}

Following closely {[}14{]} and {[}2{]}, we aim to explain here how
a classical solution to equation (\ref{E}) $(\omega,u)\in C^{1}([0,T];C^{0,\alpha}(\mathbb{R}^{2}))\times C^{1}([0,T];C_{\mathrm{loc}}^{1,\alpha}(\mathbb{R}^{2})\cap L^{\infty}(\mathbb{R}^{2}))$
with $\alpha\in(0,1)$, where $\omega(t,\cdot)$ a compactly supported
probability density on $\mathbb{R}^{2}$ and $\nabla u\in L^{\infty}([0,T];L^{\infty}\cap L^{2}(\mathbb{R}^{2}))$,
can be realized for some suitable initial data. First we recall that
given a bounded and summable vorticity $\omega$, the $L^{\infty}$
norm of the velocity field attached to $\omega$ is controlled by
the $L^{\infty}\cap L^{1}$ norm of $\omega$, as summarized in the
following simple 
\begin{lem}
Let $\omega\in L^{\infty}(\mathbb{R}^{2})\cap L^{1}(\mathbb{R}^{2})$
and let $u$ be the velocity attached to it via the Biot-Savart law,
i.e. 
\begin{equation}
u(x)\coloneqq\frac{1}{2\pi}\frac{(\cdot)^{\bot}}{|\cdot|^{2}}\star\omega(x).\label{bio savarat}
\end{equation}
 Then 
\[
||u||_{\infty}\leq C(||\omega||_{\infty}+||\omega||_{1}).
\]
\end{lem}
\begin{proof}
Let $0\leq\eta\leq1,\eta\in C_{0}^{\infty}(\mathbb{R}^{2})$ such
that $\eta\equiv1$ on $B_{1}(0)$ and $\eta\equiv0$ outside $B_{2}(0)$.
Then 
\[
|u(x)|\leq\left|\underset{\mathbb{R}^{2}}{\int}\frac{(x-y)^{\bot}\eta(x-y)\omega(y)}{|x-y|^{2}}dy|\right|+\left|\underset{\mathbb{R}^{2}}{\int}\frac{(x-y)^{\bot}(1-\eta)(x-y)\omega(y)}{|x-y|^{2}}dy\right|
\]

\[
\lesssim||\frac{\eta}{|\cdot|}\star\omega(x)||_{\infty}+||\frac{1-\eta}{|\cdot|}\star\omega||_{\infty}\lesssim||\omega||_{\infty}+||\omega||_{1}.
\]
\end{proof}
Next, below is stated a uniqueness and existence theorem for equation
(\ref{E})
\begin{thm}
\textup{\label{existence uniqueness for euler} ({[}14{]}, Theorem
8.2)} Let the initial vorticity $\omega_{0}\in L_{0}^{\infty}(\mathbb{R}^{2})$.
Then there exists a unique weak solution (in the sense of definition
8.1 in {[}14{]}) to equation (\ref{E}) with $\omega\in L^{\infty}([0,T];L_{0}^{\infty}(\mathbb{R}^{2})).$
\end{thm}
Denoting by $X^{t}(\alpha)$ the particle trajectory map of the velocity
(see 1.3 in {[}14{]}), we have 
\begin{prop}
\textup{({[}14{]}, Proposition 8.3)} Suppose that $\omega_{0}\in L^{1}(\mathbb{R}^{2})\cap L^{\infty}(\mathbb{R}^{2})$
such that $\omega_{0}\in C^{m,\alpha}(\overline{\Omega_{0}})$ where
$\Omega_{0}$ is open and bounded. Then $\omega(t,\cdot)\in C_{\mathrm{loc}}^{m,\alpha}(\Omega_{t}),u(t,\cdot)\in C_{\mathrm{loc}}^{m+1,\alpha}(\Omega_{t})$
where 

\[
\Omega_{t}\coloneqq\{X^{t}(x)|x\in\Omega_{0}\}.
\]
\end{prop}
Therefore given a compactly supported probability density $\omega_{0}\in C_{\mathrm{loc}}^{0,\alpha}(\mathbb{R}^{2})$
the unique solution provided by theorem (\ref{existence uniqueness for euler})
will be in fact a classical solution $(\omega,u)\in C^{1}([0,T],C_{\mathrm{loc}}^{0,\alpha}(\mathbb{R}^{2}))\times C^{1}([0,T],C_{\mathrm{loc}}^{1,\alpha}(\mathbb{R}^{2}))$
with compactly supported vorticity. Indeed, since $\omega(t,x)=\omega_{0}(X^{-t}(x))$,
we see that the support of $\omega(t,x)$ is included in $\Omega_{t}$
and is locally H\textcyr{\"\cyro}lder regular on $\Omega_{t}$ and
therefore this is true on the entire space $\mathbb{R}^{2}$. The
$C_{\mathrm{loc}}^{1,\alpha}(\mathbb{R}^{2})$ regularity for $u$
follows from elliptic regularity and that $X^{t}$ is a homeomorphism.
Of course, that $\omega(t,\cdot)$ is a probability density for all
times $t\in[0,T]$ is because it satisfies a transport equation (and
thus the initial sign and $L^{1}$ norm are both preserved). Since
the Gronwall estimate of section \ref{sec:Gronwall-Estimate CHAP 2}
utilizes the fact that $\nabla u(t,\cdot)$ is square summable, we
also recall the Calderon-Zygmund inequality
\begin{lem}
\textup{( {[}14{]}, Lemma 8.3)} Suppose $\omega\in L^{\infty}([0,T];L_{0}^{\infty}(\mathbb{R}^{2})$
is a weak solution to equation (\ref{E}). Then for all $1<p<\infty$
\[
||\nabla u(t,\cdot)||_{p}\leq C(||\omega_{0}||_{\infty})p.
\]
\end{lem}
When H\textcyr{\"\cyro}lder regularity on $\omega$ is available,
the limit case $p=\infty$ can also be covered via the following 
\begin{lem}
\textup{( {[}2{]}, Proposition 7.7)} Assume $\omega\in L^{a}(\mathbb{R}^{2})\cap C^{0,\alpha}(\mathbb{R}^{2})$
for some $1\leq a<\infty$ and $\alpha\in(0,1)$. Let $u$ be the
velocity attached to it. There is a constant $C$ such that 
\[
||\nabla u||_{\infty}\leq C\left(||\omega||_{a}+||\omega||_{\infty}\log(e+\frac{||\omega||_{C^{0,\alpha}}}{||\omega||_{\infty}})\right).
\]
\end{lem}

\end{document}